\documentclass{eptcs-modified}

\usepackage{amsmath}
\usepackage{amssymb}
\usepackage{url}
\usepackage{amsthm}
\usepackage{graphicx}
\usepackage{amscd}
\usepackage[usenames,dvipsnames,svgnames,table]{xcolor}

\usepackage{tikz}
\usetikzlibrary{arrows,chains,matrix,positioning,scopes}
\makeatletter
\tikzset{join/.code=\tikzset{after node path={%
\ifx\tikzchainprevious\pgfutil@empty\else(\tikzchainprevious)%
edge[every join]#1(\tikzchaincurrent)\fi}}}
\makeatother
\tikzset{>=stealth',every on chain/.append style={join},
         every join/.style={->}}
\tikzstyle{labeled}=[execute at begin node=$\scriptstyle,
   execute at end node=$]

\begin{document}

\theoremstyle{definition}
\newtheorem{theorem}{Theorem}
\newtheorem{definition}[theorem]{Definition}
\newtheorem{problem}[theorem]{Problem}
\newtheorem{assumption}[theorem]{Assumption}
\newtheorem{corollary}[theorem]{Corollary}
\newtheorem{proposition}[theorem]{Proposition}
\newtheorem{example}[theorem]{Example}
\newtheorem{lemma}[theorem]{Lemma}
\newtheorem{observation}[theorem]{Observation}
\newtheorem{fact}[theorem]{Fact}
\newtheorem{question}[theorem]{Open Question}
\newtheorem{conjecture}[theorem]{Conjecture}
\newtheorem{addendum}[theorem]{Addendum}
\newtheorem*{claim}{Claim}
\newcommand{\uint}{{[0, 1]}}
\newcommand{\Cantor}{{\{0,1\}^\mathbb{N}}}
\newcommand{\name}[1]{\textsc{#1}}
\newcommand{\me}{\name{P.}}
\newcommand{\id}{\textrm{id}}
\newcommand{\dom}{\operatorname{dom}}
\newcommand{\Dom}{\operatorname{Dom}}
\newcommand{\codom}{\operatorname{CDom}}
\newcommand{\spec}{\operatorname{spec}}
\newcommand{\opti}{\operatorname{Opti}}
\newcommand{\optis}{\operatorname{Opti}_s}
\newcommand{\Baire}{\mathbb{N}^\mathbb{N}}
\newcommand{\hide}[1]{}
\newcommand{\mto}{\rightrightarrows}
\newcommand{\Sierp}{Sierpi\'nski }
\newcommand{\BC}{\mathcal{B}}
\def\C{\textrm{C}}
\newcommand{\XC}{\textrm{XC}}
\newcommand{\CC}{\textrm{CC}}
\newcommand{\UC}{\textrm{UC}}
\newcommand{\lpo}{\textrm{LPO}}
\newcommand{\aouc}{\textrm{AoUC}}
\newcommand{\ic}[1]{\textrm{C}_{\sharp #1}}
\newcommand{\llpo}{\textrm{LLPO}}
\newcommand{\leqW}{\leq_{\textrm{W}}}
\newcommand{\leW}{<_{\textrm{W}}}
\newcommand{\equivW}{\equiv_{\textrm{W}}}
\newcommand{\equivT}{\equiv_{\textrm{T}}}
\newcommand{\geqW}{\geq_{\textrm{W}}}
\newcommand{\pipeW}{|_{\textrm{W}}}
\newcommand{\nleqW}{\nleq_\textrm{W}}
\newcommand{\leqsW}{\leq_{\textrm{sW}}}
\newcommand{\equivsW}{\equiv_{\textrm{sW}}}
\newcommand{\Sort}{\operatorname{Sort}}
\newcommand{\Contr}{\mathrm{Contr}}
\newcommand{\sort}[2]{({#1})^{\rm sort}_{#2}}
\newcommand{\pitc}{\Pi^0_2\textrm{C}}
\newcommand{\upto}{\upharpoonright}

\newcommand{\fr}{\mbox{}^\smallfrown}
\newcommand{\N}{\mathbb{N}}
\newcommand{\om}{\omega}
\newcommand{\ep}{\varepsilon}
\newcommand{\A}{\mathcal{A}}
\newcommand{\B}{\mathcal{B}}
\newcommand{\F}{\mathcal{F}}
\newcommand{\Q}{\mathcal{Q}}
\newcommand{\si}{\sigma}
\newcommand{\Afr}{\mathfrak{A}}
\newcommand{\Bfr}{\mathfrak{B}}

\title{Convex choice, finite choice and sorting}

\author{
Takayuki Kihara
\institute{Department of Mathematical Informatics\\ Nagoya University, Nagoya, Japan}
\email{kihara@i.nagoya-u.ac.jp}
\and
Arno Pauly\footnotemark
\institute{Department of Computer Science\\Swansea University, Swansea, UK\\ \& \\ Department of Computer Science\\University of Birmingham, Birmingham, UK}
\email{Arno.M.Pauly@gmail.com}
}

\def\titlerunning{Convex choice, finite choice and sorting}
\def\authorrunning{T. Kihara \& A. Pauly}
\maketitle

\begin{abstract}
We study the Weihrauch degrees of closed choice for finite sets, closed choice for convex sets and sorting infinite sequences over finite alphabets. Our main results are: One, that choice for finite sets of cardinality $i + 1$ is reducible to choice for convex sets in dimension $j$, which in turn is reducible to sorting infinite sequences over an alphabet of size $k + 1$, iff $i \leq j \leq k$. Two, that convex choice in dimension two is not reducible to the product of convex choice in dimension one with itself. Three, that sequential composition of one-dimensional convex choice is not reducible to convex choice in any dimension. The latter solves an open question raised at a Dagstuhl seminar on Weihrauch reducibility in 2015. Our proofs invoke Kleene's recursion theorem, and we describe in some detail how Kleene's recursion theorem gives rise to a technique for proving separations of Weihrauch degrees.
\end{abstract}

{\renewcommand*{\thefootnote}{*} \footnotetext{\noindent\begin{minipage}{0.1\textwidth}\includegraphics[width=\textwidth]{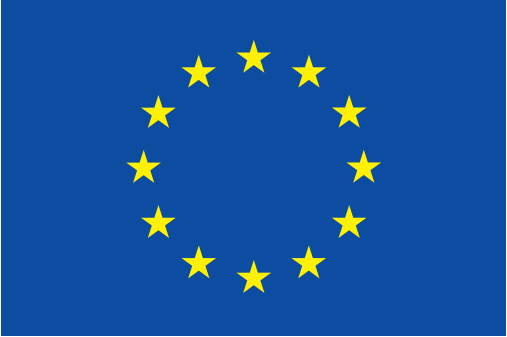}\end{minipage} \begin{minipage}{0.8\textwidth} This project has received funding from the European Unions Horizon 2020 research and innovation programme under the Marie Sklodowska-Curie grant agreement No 731143, \emph{Computing with Infinite Data}.\end{minipage}}}

\section{Introduction}
The Weihrauch degrees are the degrees of non-computability for problems in computable analysis. In the wake of work by Brattka, Gherardi, Marcone and P.~\cite{gherardi,brattka2,brattka3,paulyreducibilitylattice} they have become a very active research area in the past decade. A recent survey is found as \cite{pauly-handbook}.

We study the Weihrauch degrees of closed choice for finite sets, closed choice for convex sets and sorting infinite sequences over finite alphabets. The closed choice operators have turned out to be a useful scaffolding in that structure: We often classify interesting operations (for example linked to existence theorems) as being equivalent to a choice operator, and then prove separations for the choice operators, as they are particularly amenable for many proof techniques. Examples of this are found in \cite{brattka3,paulybrattka,paulybrattka2,hoelzl,hoelzl2,pauly-vitali,paulybrattka3,eike-neumann,pauly-marcone}. Convex choice in particular captures the degree of non-computability of finding fixed points of non-expansive mappings via the Goehde-Browder-Kirk fixed point theorem \cite{eike-neumann}.

The present article is a continuation of \cite{paulyleroux} by Le Roux and P., which already obtained some results on the connections between closed choice for convex sets and closed choice for finite sets. We introduce new proof techniques and explore the connection to the degree of sorting infinite sequences. Besides laying the foundations for future investigations of specific theorems, we are also addressing a question on the complexity caused by dimension: Researchers have often wondered whether there is a connection between the dimension of the ambient space and the complexity of certain choice principles. An initial candidate was to explore closed choice for connected subsets, but it turned out that the degree is independent of the dimension, provided this is at least 2 \cite{paulybrattka3}. As already shown in \cite{paulyleroux}, this works for convex choice. One reason for this was already revealed in \cite{paulyleroux}: We need $n$ dimensions in order to encode a set of cardinality $n + 1$. We add another reason here: Each dimension requires a separate instance of sorting an infinite binary sequence in order to find a point in a convex set.

Our Theorem \ref{maintheorem2} gives a negative answer to \cite[Question 3.14]{paulyleroux}, which was again raised as an open problem at a Dagstuhl seminar on Weihrauch reducibility, c.f.~\cite{pauly-dagstuhl}.

An extended abstract lacking Sections \ref{sec:triangles} and \ref{sec:convexcomposition} appeared as \cite{pauly-kihara5-tamc}.

\paragraph*{Structure of the paper} Most of our results are summarized in Figure \ref{figure:reducibilities} on Page \pageref{figure:reducibilities}. Section \ref{sec:background} provides a brief introduction to Weihrauch reducibility. In Section \ref{sec:definitions} we provide formal definitions of the principles under investigation, and give a bit more context. We proceed to introduce our new technique to prove separations between Weihrauch degrees in Section \ref{sec:recursiontheorem}; it is based on Kleene's recursion theorem. The degree of sorting an infinite binary sequence is studied in Section \ref{sec:sort}, including a separation technique adapted specifically for this in Subsection \ref{subsec:displacement}, its connection to convex choice in Subsection \ref{subsec:sortconvex} and a digression on the task of finding connected components of countable graphs in Subsection \ref{subsec:sortdigression}. Section \ref{sec:finitesort} is constituted by Theorem \ref{theo:finitebelowsort} and its proof, establishing the precise relationship between finite choice and sorting. In Section \ref{sec:game} we introduce a game characterizing reducibility between finite choice for varying cardinalities. Section \ref{sec:triangles} develops a technique to prove that $\XC_2 \nleqW \XC_1 \times \XC_1$, and in Section \ref{sec:convexcomposition} we show ${\sf XC}_1\star{\sf AoUC}\not\leq_W{\sf XC}_k$ for all $k\in\mathbb{N}$.

\section{Background on Weihrauch reducibility}
\label{sec:background}

Weihrauch reducibility is a quasiorder defined on multi-valued functions between represented spaces. We only give the core definitions here, and refer to \cite{pauly-synthetic} for a more in-depth treatment. Other sources for computable analysis are \cite{weihrauchd,brattkaintro}.

\begin{definition}
A represented space $\mathbf{X}$ is a set $X$ together with a partial
surjection $\delta_{\mathbf{X}} : \subseteq \Baire \to X$.
\end{definition}

A partial function $F : \subseteq \Baire \to \Baire$ is called a \emph{realizer}
of a function $f : \subseteq \mathbf{X} \to \mathbf{Y}$ between represented
spaces, if $f(\delta_\mathbf{X}(p)) = \delta_\mathbf{Y}(F(p))$ holds for all
$p \in \dom(f \circ \delta_\mathbf{X})$. We denote $F$ being an realizer of
$f$ by $F \vdash f$. We then call $f : \subseteq \mathbf{X} \to \mathbf{Y}$
\emph{computable} (respectively \emph{continuous}), iff it has a computable
(respectively continuous) realizer.

Represented spaces can adequately model most spaces of interest in \emph{everyday mathematics}. For our purposes, we are primarily interested in the construction of the hyperspace of closed subsets of a given space.

The category of represented spaces and continuous functions is
cartesian-closed, by virtue of the UTM-theorem. Thus, for any two represented
spaces $\mathbf{X}$, $\mathbf{Y}$ we have a represented spaces
$\mathcal{C}(\mathbf{X},\mathbf{Y})$ of continuous functions from
$\mathbf{X}$ to $\mathbf{Y}$. The expected operations involving
$\mathcal{C}(\mathbf{X},\mathbf{Y})$ (evaluation, composition, (un)currying)
are all computable.

Using the Sierpi\'nski space $\mathbb{S}$ with underlying
set $\{\top,\bot\}$ and representation $\delta_{\mathbf{S}} : \Baire \to
\{\top,\bot\}$ defined via $\delta_{\mathbf{S}}(\bot)^{-1} = \{0^\omega\}$,
we can then define the represented space $\mathcal{O}(\mathbf{X})$ of
\emph{open} subsets of $\mathbf{X}$ by identifying a subset of $\mathbf{X}$
with its (continuous) characteristic function into $\mathbb{S}$. Since
countable \emph{or} and binary \emph{and} on $\mathbb{S}$ are computable, so
are countable union and binary intersection of open sets.

The space
$\mathcal{A}(\mathbf{X})$ of closed subsets is obtained by taking formal
complements, i.e.~the names for $A \in \mathcal{A}(\mathbf{X})$ are the same
as the names of $X \setminus A \in \mathcal{O}(\mathbf{X})$ (i.e.~we are
using the negative information representation). Intuitively, this means that when reading a name for a closed set, this can always shrink later on, but never grow. It is often very convenient that we can alternatively view $A \in \mathcal{A}(\Cantor)$ as being represented by some tree $T$ via $[T] = A$ (here $[T]$ denotes the set of infinite paths through $T$).

We can now define Weihrauch reducibility. Again, we give a very brief treatment here, and refer to \cite{pauly-handbook} for more details and references.

\begin{definition}[Weihrauch reducibility]
\label{def:weihrauch} Let $f,g$ be multivalued functions on represented
spaces. Then $f$ is said to be {\em Weihrauch reducible} to $g$, in symbols
$f\leqW g$, if there are computable functions $K,H:\subseteq\Baire\to\Baire$ such
that $\left(p \mapsto K\langle p, GH(p) \rangle \right )\vdash f$ for all $G
\vdash g$.
\end{definition}

The Weihrauch degrees (i.e.~equivalence classes of $\leqW$) form a distributive lattice, but we will not need the lattice operations in this paper. Instead, we use two kinds of products. The usual cartesian product induces an operation $\times$ on Weihrauch degrees. We write $f^k$ for the $k$-fold cartesian product with itself. The compositional product $f \star g$ satisfies that $$f \star g \equivW \max_{\leqW} \{f_1 \circ g_1 \mid f_1 \leqW f \wedge g_1 \leqW g\}$$
and thus is the hardest problem that can be realized using first $g$, then something computable, and finally $f$. The existence of the maximum is shown in \cite{paulybrattka4} via an explicit construction, which is relevant in some proofs. Both products as well as the lattice-join can be interpreted as logical \emph{and}, albeit with very different properties.

We'll briefly mention a further unary operation on Weihrauch degrees, the
finite parallelization $f^*$. This
has as input a finite tuple of instances to $f$ and needs to solve all of
them.

As mentioned in the introduction, the closed choice principles are valuable benchmark degrees in the Weihrauch lattice:
\begin{definition}
For a represented space $\mathbf{X}$, the closed choice principle $\C_\mathbf{X} : \subseteq \mathcal{A}(\mathbf{X}) \mto \mathbf{X}$ takes as input a non-empty closed subset $A$ of $\mathbf{X}$ and outputs some point $x \in A$.
\end{definition}

\section{The principles under investigation}
\label{sec:definitions}
We proceed to give formal definitions of the three problems our investigation is focused on. These are \emph{finite choice}, the task of selecting a point from a closed subset (of $\Cantor$ or $\uint^n$) which is guaranteed to have either exactly or no more than $k$ elements; \emph{convex choice}, the task of selecting a point from a convex closed subset of $\uint^k$; and \emph{sorting} an infinite sequence over the alphabet $\{0,1,\ldots,k\}$ in increasing order. Our main result is that each task forms a strictly increasing chain in the parameter $k$, and these chains are perfectly aligned as depicted in Figure \ref{figure:reducibilities}. For finite choice and convex choice, this was already established in \cite{paulyleroux}. Our Theorem \ref{theo:finitebelowsort} implies the main theorem from \cite{paulyleroux} with a very different proof technique.

\begin{definition}[{\cite[Definition 7]{paulyleroux}}]
For a represented space $\mathbf{X}$ and $1 \leq n \in \mathbb{N}$, let $\C_{\mathbf{X},\sharp=n} := \C_{\mathbf{X}}|_{\{A \in \mathcal{A}(\mathbf{X}) \mid |A| = n\}}$ and $\C_{\mathbf{X},\sharp\leq n} := \C_{\mathbf{X}}|_{\{A \in \mathcal{A}(\mathbf{X}) \mid 1 \leq |A| \leq n\}}$.
\end{definition}

It was shown as \cite[Corollary 10]{paulyleroux} that for every computably compact computably rich computable metric space $\mathbf{X}$ we find $\C_{\mathbf{X},\sharp=n} \equivW \C_{\Cantor,\sharp=n}$ and $\C_{\mathbf{X},\sharp\leq n} \equivW \C_{\Cantor,\sharp \leq n}$. This in particular applies to $\mathbf{X} = \uint^d$. We denote this Weihrauch degree by $\ic{= n}$ respectively $\ic{\leq n}$.

\begin{definition}[{\cite[Definition 8]{paulyleroux}}]
By $\XC_n$ we denote the restriction of $\C_{\uint^n}$ to convex sets.
\end{definition}

Since for subsets of $\uint$ being an interval, being convex and being connected all coincide, we find that $\XC_1$ is the same thing as one-dimensional connected choice $\CC_1$ as studied in \cite{paulybrattka3} and as interval choice $\C_{\mathrm{I}}$ as studied in \cite{brattka3}.

\begin{definition}
Let $\mathrm{Sort}_d : d^\omega \to d^\omega$ be defined by $\mathrm{Sort}_d(p) = 0^{c_0}1^{c_1}\ldots k^\infty$, where $|\{n \mid p(n) = 0\}| = c_0$, $|\{n \mid p(n) = 1\}| = c_1$, etc, and $k$ is the least such that $|\{n \mid p(n) = k\}| = \infty$. We write just $\Sort$ for $\Sort_2$.
\end{definition}

$\mathrm{Sort}$ was introduced and studied in \cite{paulyneumann}, and then generalized to $\mathrm{Sort}_k$ in \cite{hoelzl2}. Note that the principle just is about sorting a sequence in order without removing duplicates. In \cite{paulytsuiki-arxiv} it is shown that $\mathrm{Sort}_{n+1} \equivW \mathrm{Sort}^n$; it follows that $\Sort^* \equivW \Sort_d^* \equivW \coprod_{d \in \mathbb{N}} \Sort_d$. The degree $\Sort^*$ was shown in \cite{paulyneumann} to capture the strength of the strongly analytic machines \cite{hotz2,gaertnerhotz}, which in turn are an extension of the BSS-machines \cite{blum2}. $\Sort$ is equivalent to Thomae's function; and to the translation of the standard representation of the reals into the continued fraction representation \cite{weihrauchb}. In \cite{pauly-marcone}, $\Sort$ is shown to be equivalent to certain projection operators.

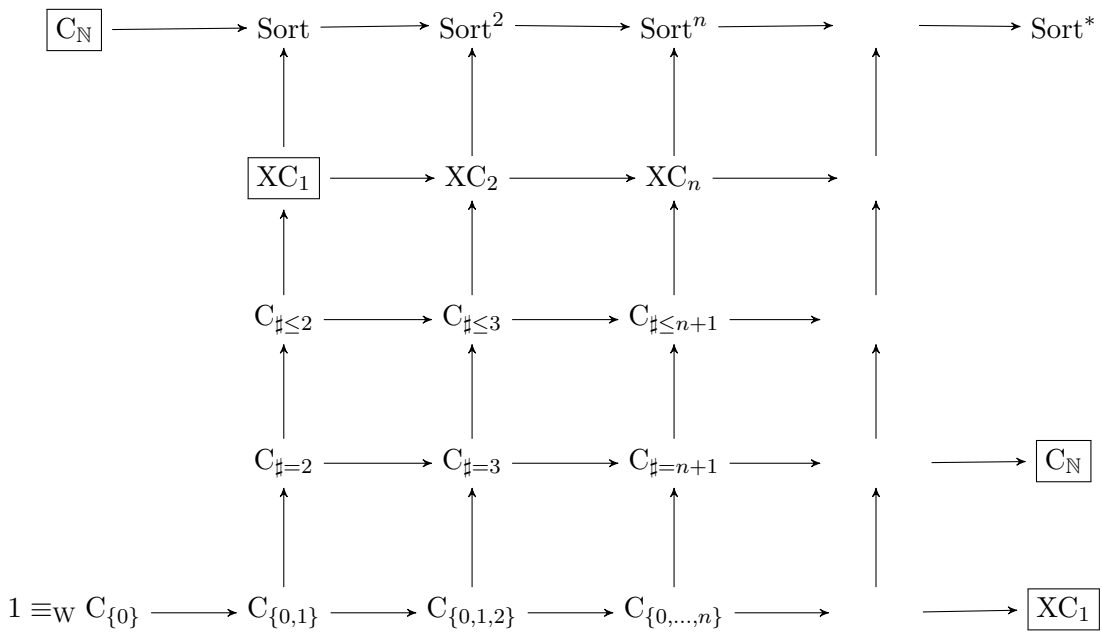
\begin{figure}[htbp]
\begin{tikzpicture}
  \matrix (m) [matrix of math nodes, row sep=3em, column sep=3em]
    { \boxed{\C_\mathbb{N}} & \Sort  & \Sort^2  & \Sort^n  & \phantom{\Sort^4} & \Sort^* \\
    \phantom{\XC_n} & \boxed{\XC_1}  & \XC_2  & \XC_n  & \phantom{\XC_n} &   \\
     & \ic{\leq2}  & \ic{\leq3}  & \ic{\leq n+1}  &  \phantom{\ic{\leq n+1}} &  \\
       & \ic{=2}  & \ic{=3}  & \ic{=n+1}  & \phantom{\ic{=n+1}} & \boxed{\C_\mathbb{N}}\\
      1 \equivW \C_{\{0\}} & \C_{\{0, 1\}} & \C_{\{0, 1, 2\}} & \C_{\{0, \ldots, n\}} & \phantom{\C_{\{0, 1\}}} &  \boxed{\XC_1}\\ };
  { [start chain] \chainin (m-1-1);
      \chainin (m-1-2);
    \chainin (m-1-3);
    \chainin (m-1-4);
    \chainin (m-1-5);
        \chainin (m-1-6);}
    { [start chain] \chainin (m-2-2);
    \chainin (m-2-3);
    \chainin (m-2-4);
    \chainin (m-2-5); }
    { [start chain] \chainin (m-3-2);
    \chainin (m-3-3);
    \chainin (m-3-4);
    \chainin (m-3-5); }
 { [start chain] \chainin (m-4-2);
    \chainin (m-4-3);
    \chainin (m-4-4);
    \chainin (m-4-5); \chainin (m-4-6);}
      { [start chain]
     \chainin (m-5-1);
    \chainin (m-5-2);
      { [start branch=A] \chainin (m-4-2); \chainin (m-3-2); \chainin (m-2-2); \chainin (m-1-2);}
    \chainin (m-5-3);
      { [start branch=B] \chainin (m-4-3); \chainin (m-3-3); \chainin (m-2-3); \chainin (m-1-3);}
    \chainin (m-5-4);
           { [start branch=C] \chainin (m-4-4); \chainin (m-3-4); \chainin (m-2-4); \chainin (m-1-4);}
    \chainin (m-5-5);
          { [start branch=D] \chainin (m-4-5); \chainin (m-3-5); \chainin (m-2-5); \chainin (m-1-5);}
    \chainin (m-5-6); }

\end{tikzpicture}
\caption{Overview of our results; extending \cite[Figure 1]{paulyleroux} by the top row.
The diagram depicts all Weihrauch reductions between the stated principles up to transitivity. Boxes mark degrees appearing in two places in the diagram. Our additional results are provided as Theorems \ref{theo:convexsort} and \ref{theo:finitebelowsort}.}
\label{figure:reducibilities}
\end{figure}

There are some additional Weihrauch problems we make passing reference to. \emph{All-or-unique choice} captures the idea of a problem either having a unique solution, or being completely undetermined:

\begin{definition}
Let $\aouc_\mathbf{X}$ be the restriction of $\C_\mathbf{X}$ to $\{\{x\} \mid x \in \mathbf{X}\} \cup \{X\}$.
\end{definition}

 A prototypical example (which is equivalent to the full problem) is solving $ax = b$ over $\uint$ with $0 \leq b \leq a$: Either there is the unique solution $\frac{b}{a}$, or $b = a = 0$, and any $x \in \uint$ will do. The degree of $\aouc_\mathbf{X}$ is the same for any computably compact computably rich computable metric space, in particular for $\mathbf{X} = \Cantor$ or $\mathbf{X} = \uint^d$. We just write $\aouc$ for that degree. This problem was studied in \cite{paulyincomputabilitynashequilibria,pauly-kihara2-mfcs} where it is shown that $\aouc^*$ is the degree of finding Nash equilibria in bimatrix games and of executing Gaussian elimination.

\section{Proving separations via the recursion theorem}
\label{sec:recursiontheorem}
A core technique we use to prove our separation results invokes Kleene's recursion theorem in order to let us prove a separation result by proving computability of a certain map (rather than having to show that no computable maps can witness a reduction). We had already used this technique in \cite{pauly-kihara2-mfcs}, but without describing it explicitly. Since the technique has proven very useful, we formally state the argument here as Theorem \ref{theo:recursionweihrauch} after introducing the necessary concepts to formulate it.

\begin{definition}
A representation $\delta$ of $\mathbf{X}$ is \emph{precomplete}, if every computable partial $f :\subseteq 2^\omega \to \mathbf{X}$ extends to a computable total $F : 2^\omega \to \mathbf{X}$.
\end{definition}

\begin{proposition}
For effectively countably-based $\mathbf{X}$, the space $\mathcal{O}(\mathbf{X})$ (and hence $\mathcal{A}(\mathbf{X})$) is precomplete.
\begin{proof}
It suffices to show this for $\mathcal{O}(\mathbb{N})$, where it just follows from the fact that we can delay providing additional information about a set as long as we want; and will obtain a valid name even if no additional information is forthcoming.
\end{proof}
\end{proposition}

The preceding proposition is a special case of \cite[Theorem 6.5]{selivanov5}, which shows that many pointclasses have precomplete representations.

\begin{proposition}
The subspaces of $\mathcal{A}(\uint^n)$ consisting of the connected respectively the convex subsets are computable multi-valued retracts, and hence precomplete.
\begin{proof}
For the connected sets, this follows from \cite[Proposition 3.4]{paulybrattka3}; for convex subsets this follows from computability of the convex hull operation on $\uint^n$, see e.g.~\cite[Proposition 1.5]{paulyleroux} or \cite{ziegler8}.
\end{proof}
\end{proposition}

By $\mathcal{M}(\mathbf{X},\mathbf{Y})$ we denote the represented space of strongly continuous multivalued functions from $\mathbf{X}$ to $\mathbf{Y}$ studied in \cite{paulybrattka4}. The precise definition of strong continuity is irrelevant for us, we only need every partial continuous function on $\Cantor$ induces a minimal strongly continuous multivalued function that it is a realizer of; and conversely, every strongly continuous multivalued function is given by a continuous partial realizer.

\begin{theorem}
\label{theo:recursionweihrauch}
Let $\mathbf{X}$ have a total precomplete representation. Let $f : \mathbf{X} \mto \mathbf{Y}$ and $g : \mathbf{U} \mto \mathbf{V}$ be such that there exists a computable $e : \mathbf{U} \times \mathcal{M}(\mathbf{V},\mathbf{Y}) \mto \mathbf{X}$ such that if $x \in e(u,k)$ and $v \in g(u)$, then $k(v) \nsubseteq f(x)$. Then $f \nleqW g$.
\end{theorem}
\begin{proof}
Assume that $f \leqW g$ via computable $H$, $K$. Let computable $E$ be a realizer of $e$. Let $(\phi_ n: \subseteq \mathbb{N} \to \mathbb{N})_{n \in \mathbb{N}}$ be a standard enumeration of the partial computable functions. By assumption, we can consider each $\phi_n$ to denote some element in $\mathbf{X}$. Let $\lambda$ be a computable function such that $\phi_{\lambda(n)} = E(H(\phi_n), (v \mapsto K(\phi_n,v)))$. By Kleene's fixed point theorem, there is some $n_0$ with $\phi_{n_0} = \phi_{\lambda(n_0)}$. Inputting $\phi_{n_0}$ to $f$ fails the assumed reduction witnesses.
\end{proof}

Theorem \ref{theo:recursionweihrauch} says that, in order to show that $f \nleqW g$, it suffices to describe a computable strategy that, given a $g$-instance $u$ and an outer reduction $k$, produces an $f$-instance $x$ witnessing that $k$ fails to give a solution to $f(x)$ from a solution to $g(u)$.

As simple sample application for how to prove separations of Weihrauch degrees via the recursion theorem, we shall point out that $\XC_1$ already cannot solve some simple products. For contrast, however, note that $\C_2^* \leqW \XC_1$ was shown as \cite[Proposition 9.2]{paulybrattka3}.

\begin{theorem}
\label{theo:simpleproducts} $\C_2 \times \aouc \nleqW \XC_1$.
\end{theorem}
\begin{proof}
Given a convex tree $T \subseteq 2^{<\om}$ and a partial continuous function
$\phi : \subseteq \Cantor \to 2 \times \Cantor$, we compute set $S \in \mathcal{A}(\{0,1\})$ and $V \in \mathcal{A}(\Cantor)$ such that $S \neq \emptyset$, and $V = \Cantor$ or $V = \{p\}$ for some $p \in \Cantor$. Our construction ensures that $\exists p \in [T] \ \phi(p) \notin S \times V$.

Initially, $S = \{0,1\}$ and $V = \Cantor$.

We first search for $s$ such that for any $\sigma \in T$ of length $s$, the first value of $\phi(\sigma)$ is determined. If we never find one, then $S = \{0,1\}$ and $V = \Cantor$ work as desired.

Next, we search for some $\tau \in \{0,1\}^s$ such that $P_\tau := [T] \cap \bigcup_{j < 2} \phi^{-1}(j,[\tau])$ is such that any interval contained in $P_\tau$ is contained in some $[\sigma]$ for $\sigma \in \{0,1\}^s$.  Note that if $(J_i)_{i \in I}$ is a collection of pairwise disjoint intervals in $\Cantor$ such that
every $J_i$ intersects with at least two cylinders $[\sigma]$ and $[\sigma']$ for some strings $\sigma \neq \sigma'$ of
length $s$, then the size of $I$ is at most $2^s - 1$. Hence, if $\phi$ is defined on $[T]$, such a $\tau$ has to exist. Once we have found it, we set $V = \{\tau0^\omega\}$.

Either we are already done (since we would have that $\exists p \in [T] \ \phi(p) \notin S \times V$), or it holds that $[T] \subseteq [\sigma]$ for some $\sigma \in \{0,1\}^s$. In that case, by choice of $s$ we find that $\exists j \in \{0,1\} \ \pi_0\phi(p) = j$ for all $p \in [T]$. We can set $S = \{1-j\}$, and have obtained the desired property that $\exists p \in [T] \ \phi(p) \notin S \times V$. By Theorem \ref{theo:recursionweihrauch} with $e\colon(T,\phi)\mapsto(S,V)$, the claim follows.
\end{proof}

\section{Some observations on $\Sort$}
\label{sec:sort}
\subsection{Displacement principle for $\Sort_k$}
\label{subsec:displacement}
The basic phenomenon that the number of parallel copies of $\Sort$ being used corresponds to a dimensional feature can already by a result similar in feature to the displacement principle from \cite{paulybrattka3}:

\begin{proposition}
$\C_2 \times f \leqW \Sort_{k+1}$ implies $f \leqW \Sort_k \times \C_\mathbb{N}$.
\begin{proof}
Let the reduction $\C_2 \times f \leqW \Sort_{k+1}$ be witnessed by computable $H$, $K_1$, $K_2$. Assume, for the sake of a contradiction, that for some input $x$ to $f$ and a name $p$ for $\{0,1\}$ it holds that $H(p,x)$ contains infinitely many $0$s. In that case, $\Sort_k(H(p,x)) = 0^\omega$, and hence $K_1$ is defined as either $0$ or $1$ on $p, x, 0^\omega$. But then there is some $k \in \mathbb{N}$ such that $K_1$ already outputs the answer on reading some prefix $p_{\leq k}, x_{\leq k}, 0^k$. Additionally, we can chose some $k' \geq k$ such that $H$ writes at least $k'$ $0$s upon reading the prefixes $p_{\leq k'}, x_{\leq k'}$. By changing $p$ after $k'$ to be a name of $\{1 - K_1(p,x,0^\omega)\}$ shows the contradiction.

Now we note that $x \mapsto H(p,x)$ and $K_2$ witness a reduction from $f$ to the restriction of $\Sort_{k+1}$ to inputs containing only finitely many $0$s. But this restriction is reducible to $\Sort_k \times \C_\mathbb{N}$: In parallel, call $\Sort_k$ on the sequence obtained by skipping $0$s and decrementing every other digit by $1$, and using $\C_\mathbb{N}$ to determine the original number of $0$s.
\end{proof}
\end{proposition}

\begin{corollary}
Let $f$ be a closed fractal. Then $\C_2 \times f \leqW \Sort_{k+1}$ implies $f \leqW \Sort_k$.
\end{corollary}

\begin{corollary}
$\C_2 \times \C_{\sharp \leq 2}^n \nleqW \Sort_{n+1}$.
\end{corollary}

\begin{corollary}
$\C_2 \times \XC_1^n \nleqW \Sort_{n+1}$
\end{corollary}

We also get an alternative proof of the following, which was previously shown in \cite{paulyneumann} using the squashing principle from \cite{shafer}:
\begin{corollary}
$\Sort_{k+1} \nleqW \Sort_k$
\end{corollary}

\subsection{$\Sort$ and convex choice}
\label{subsec:sortconvex}
The one-dimensional case of the following theorem was already proven as \cite[Proposition 16]{hoelzl2}:
\begin{theorem}
\label{theo:convexsort}
$\XC_n \leqW \Sort_{n+1}$
\begin{proof}
Let $(H_i^d)_{i \in \mathbb{N}}$ be an effective enumeration of the $d$-dimensional rational hyperplanes for each $d \leq n-1$. Given $A \in \mathcal{A}(\uint^n)$, we can recognize that $A \cap H_i^d = \emptyset$ by compactness of $\uint^n$. We proceed to compute an input $p$ to $\Sort_{n+1}$ as follows:

We work in stages $(\ell_0,\ldots,\ell_{n-1})$. We simultaneously test whether $A \cap H_{\ell_0}^{n-1} = \emptyset$, whether $A \cap H_{\ell_0}^{n-1} \cap H_{\ell_1}^{n-2} = \emptyset$, $\ldots$, and whether $A \cap H_{\ell_0}^{n-1} \cap \ldots \cap H_{\ell_{n-1}}^{1} = \emptyset$.

If we find a confirmation for a query involving $\ell_k$ as the largest index, we write a $k$ to $p$, increment $\ell_k$ by $1$, and reset any $\ell_i$ for $i > k$. All tests of smaller indices are continued (and hence will eventually fire if true before a largest index test interferes). In addition, we write $n$s to $p$ all the time to ensure an infinite result.

Now consider the output $\mathrm{Sort}_{n+1}(p)$. If this is $0^\omega$, then $A$ does not intersect any $n-1$-dimensional rational hyperplane at all. As a convex set, $A$ has to be a singleton. Thus, as long as we read $0$s from $\mathrm{Sort}_{n+1}(p)$, we can just wait until $A$ shrinks sufficiently to produce the next output approximation. If we ever read a $1$ in $\mathrm{Sort}_{n+1}(p)$ at position $t$, we have thus found a $n-1$-dimensional hyperplane $H_t^{n-1}$ intersecting $A$. We can compute $A \cap H_t^{n-1} \in \mathcal{A}(\uint^n)$, and proceed to work with that set. By retracing the computation leading up to the observation that $A \cap H_{t-1}^{n-1} = \emptyset$, we can find out how many larger-index tests were successful before that. We disregard their impact on $\Sort_{n+1}(p)$. Now as long as we keep reading $1$s, we know that $A \cap H_{t-1}^{n-1}$ is not intersecting $n-2$-dimensional rational hyperplanes (and hence could be singleton). Finding a $2$ means we have identified a $n-2$-dimensional hyperplane $H_{t'}^{n-2}$ intersecting $A \cap H_{t-1}^{n-1}$, and we proceed to work with $A \cap H_{t-1}^{n-1} \cap H_{t'}^{n-2}$. Continuing this process, we always find that either our set has been collapsed to singleton (from which we can extract the point), or we will be able to reduce its dimension further (which can happen only finitely many times).
\end{proof}
\end{theorem}

\subsection{A digression: $\Sort$ and finding connected components of a graph}
\label{subsec:sortdigression}
On a side note, we explore how $\Sort$ relate to the problem $\textrm{FCC}$ of finding a connected component of a countable graph with only finitely many connected components. Here the graph $(V,E)$ is given via the characteristic functions of $V \subseteq \mathbb{N}$ and $E \subseteq \mathbb{N} \times \mathbb{N}$, and the connected component is to be produced likewise as its characteristic function. In addition, we have available to us an upper bound for the number of its connected components. In the reverse math context, this problem was studied in \cite{mummert3} and shown to be equivalent to $\Sigma^0_2$-induction.

\begin{theorem}
The following are equivalent:
\begin{enumerate}
\item $\textrm{FCC}$
\item $\Sort^*$
\end{enumerate}
\end{theorem}
\begin{proof}
\begin{description}
\item[$\textrm{FCC} \leqW \coprod_{k \in \mathbb{N}} \Sort_k$] \

We are given $n \in \mathbb{N}$ and a graph with at most $n$ connected components. For each $2 \leq i \leq n$, we pick some standard enumeration $(V_j^i)_{j \in \mathbb{N}}$ of the $i$-element subsets of $\mathbb{N}$. As soon as we learn that none of the $V_j^i$ with $j \leq l$ is an independent set, we write the $l$-th symbol $i - 2$ on the input to $\Sort_{n-1}$. We write an $n-1$ occasionally to ensure that the output is actually infinite.

Now assume we have access to the corresponding output $q$ of $\Sort_{n-1}$. This will be $0^\omega$ iff the graph had a single connectedness component, and of the form $0^l1p$ else where $V_l^2$ is an independent pair. We can thus start computing the connectedness component of $0$ by searching in parallel whether $q \neq 0^\omega$ and searching for a path from $0$ to the current number. Either search will terminate. In the latter case, we can answer yes. In the former, we now search for paths to the two vertices in the pair (and thus might be answer to correctly no). Simultaneously we investigate the remnant $p$ whether $p = 1^\omega$ (and thus the graph has $2$ connectedness components, and any vertex is linked to either member of $V^2_l$), or find an independent set of size $3$, etc.

\item[$\Sort_k \equivW \Sort^{k-1}$] \

This was shown in \cite{paulytsuiki-arxiv}.

\item[$\Sort \leqW \textrm{FCC}$] \

We compute a graph with at most $2$ connectedness components. The graph will be bipartite, with the odd and even numbers being separate components. All odd numbers are connected to $0$, and at any stage there will be some even number $2n$ not yet connected to $0$, which represents some number $i$ such that we have not yet read $i$ times $0$ in the input $p$ to $\Sort$. If we read the $i$-th $0$ in $p$ at time $t$, we connect $2t+1$ to both $0$ and $2n$. If we read a $1$ at time $t$, then $2t+1$ gets connected to $0$ and $2t$.

If $p$ contains infinitely many $0$s, then we end up with a single connectedness component. Otherwise we obtain either the connectedness component of $0$, or equivalently, its complement. Once we see that e.g.~$2$ is in this connectedness component, then we can output $0$. Moreover, then $2$ must be linked to $0$ via some $2t+1$ (which we can exhaustively search for), and whether $2t$ is in the connectedness component tells us whether the next bit of the output is $1$ (and then continuous as $1^\omega$), or $0$ again, in which case we need to search for the next significant digit.

\item[$\textrm{FCC} \times \textrm{FCC} \leqW \textrm{FCC}$] \

Just use the product graph.
\end{description}
\end{proof}

\section{Finite choice and sorting}
\label{sec:finitesort}

\begin{theorem}
\label{theo:finitebelowsort}
${\sf C}_{\#\leq k+1}\not\leq_{\sf W}{\sf Sort}_k$.
\end{theorem}

\begin{proof}
By Theorem \ref{theo:recursionweihrauch}, it suffices to describe an effective procedure which, given $\alpha\in k^\om$ and $\Phi$, constructs an instance $C$ of ${\sf C}_{\#\leq k+1}$ such that there is a solution $q$ to ${\sf Sort}_k(\alpha)$ such that $\Phi(q)$ is not a solution to $C$.
(Apply Theorem \ref{theo:recursionweihrauch} to $e\colon(\alpha,\Phi)\mapsto C$.)

For a finite tree $T$ of height $s$, we say that $\sigma\in T$ is {\em extendible} if there is a leaf $\rho\in T$ of height $s$ which extends $\sigma$.
Note that an instance of ${\sf C}_{\#\leq k+1}$ is generated by an increasing sequence $(T_s)_{s\in\om}$ of finite binary trees satisfying the following conditions for every $s$.
\begin{enumerate}
\item[(I)] $T_s$ is of height $s$, and  $T_s$ has at least one, and at most $k+1$ extendible leaves.
\item[(II)] Every node $\sigma\in T_{s+1}\setminus T_s$ is of length $s+1$, and extends an extendible leaf of $T_s$.
\end{enumerate}
More precisely, for such a sequence $(T_s)$, the union $T=\bigcup_sT_s$ forms a $(T_s)$-computable tree which has at most $k+1$ many infinite paths.
Therefore, the set of all infinite paths $C=[T]$ through $T$ is an instance of ${\sf C}_{\#\leq k+1}$.

For $\eta\in k^{<\om}$ and $u<k$, let $N[\eta,u]$ be the number of the occurrences of $u$'s in $\eta$, i.e., $N[\eta,u]=\#\{i:\eta(i)=u\}$.
We define the {\em $u$-partial sort of $\eta$} as the following string:
\[\sort{\eta}{u}=0^{N[\eta,0]}1^{N[\eta,1]}2^{N[\eta,2]}\dots (u-1)^{N[\eta,u-1]}.\]

Our description of an effective procedure which, given an instance $\alpha$ of ${\sf Sort}_k$, returns a sequence $(T_s)_{s\in\om}$ of finite trees generating an instance of ${\sf C}_{\#\leq k+1}$ is subdivided into $k$ many strategies $(\mathcal{S}_u)_{u<k}$.
At stage $s$, the $u$-th strategy $\mathcal{S}_u$ for $u<k$ believes that $u$ is the least number occurring infinitely often in a given instance $\alpha$ of ${\sf Sort}_{k}$, and there is no $i\geq s$ such that $\alpha(i)<u$.
In other words, the strategy $\mathcal{S}_u$ believes that $\sort{\alpha\upto s}{u}\fr u^\om$, the $u$-partial sort of the current approximation of $\alpha$ followed by the infinite constant sequence $u^\om$, is the right answer to the instance $\alpha$ of ${\sf Sort}_k$.
Then, the strategy $\mathcal{S}_u$ waits for $\Phi(\sort{\alpha\upto s}{u}\fr u^\om)$ being a sufficiently long extendible node $\rho$ of $T_s$, and then make a branch immediately after an extendible leaf $\rho_u\in T_s$ extending $\rho$, where this branch will be used for diagonalizing $\Phi(\sort{\alpha\upto s}{u}\fr u^\om)$.
This action injures all lower priority strategies $(\mathcal{S}_v)_{u<v<k}$ by initializing their states and letting $\rho_v$ be undefined.

More precisely, each strategy $\mathcal{S}_u$ has a state, ${\tt state}_s(u)\in\{0,1,2\}$, at each stage $s$, which is initialized as ${\tt state}_0(u)=0$.
We also define a partial function $u\mapsto \rho^s_u$ for each $s$, where $\rho^s_u$ is extendible in $T_s$ if it is defined.
Roughly speaking, $\rho^s_u$ is the stage $s$ approximation of the diagonalize location for the $u$-th strategy as described above.
We assume that $\rho^0_u$ is undefined for $u>0$,  for any $s\in\om$, $\rho^s_0$ is defined as an empty string, and $\rho^s_u$ is a finite string whenever it is defined.

At the beginning of stage $s+1$, inductively assume that a finite tree $T_s$ of height $s$ and a partial function $u\mapsto\rho^s_u$ has already been defined.
Moreover, we inductively assume that if ${\tt state}_s(u)=1$ then $\rho^s_u$ is defined, and $\rho^s_u\fr i$ is extendible in $T_s$ for each $i<2$.
At substage $u$ of stage $s+1$, the strategy $\mathcal{S}_u$ acts as follows:
\begin{enumerate}
\item If $\sort{\alpha\upto s+1}{u}\not=\sort{\alpha\upto s}{u}$, then initialize the strategy, that is, put ${\tt state}_{s+1}(u)=0$, and let $\rho^{s+1}_u$ be undefined.
Then go to the next substage $u+1$ if $u<k$; otherwise go to the next stage $s+2$.
\item If $\sort{\alpha\upto s+1}{u}=\sort{\alpha\upto s}{u}$ and ${\tt state}_s(u)=0$, then ask if $\Phi(\sort{\alpha\upto s}{u}\fr u^\om)[s]$ is an extendible node $\rho\in T_s$ such that for any $v<u$, if $\rho^s_v$ is defined, then $\rho\not\preceq\rho^s_v$ holds.
\begin{enumerate}
\item If yes, define $\rho^{s+1}_u$ as the leftmost extendible leaf of $T_s$ extending such a $\rho$, and put ${\tt state}_{s+1}(u)=1$.
Injure all lower priority strategies, that is, put ${\tt state}_{s+1}(v)=0$ and let $\rho^{s+1}_v$ be undefined for any $u<v<k$.
Then go to the next stage $s+2$.
\item If no, go to the next substage $u+1$ if $u<k$; otherwise go to the next stage $s+2$.
\end{enumerate}
\item If $\sort{\alpha\upto s+1}{u}=\sort{\alpha\upto s}{u}$ and ${\tt state}_s(u)=1$, then ask if $\Phi(\sort{\alpha\upto s}{u}\fr u^\om)[s]$ is an extendible node $\rho\in T_s$ which extends $\rho^s_u\fr i$ for some $i<2$.
\begin{enumerate}
\item If yes, define $\rho^{s+1}_u=\rho^s_u\fr(1-i)$ for such $i$, and put ${\tt state}_{s+1}(u)=2$.
Injure all lower priority strategies, that is, put ${\tt state}_{s+1}(v)=0$ and let $\rho^{s+1}_v$ be undefined for any $u<v<k$.
Then go to the next stage $s+2$.
\item If no, go to the next substage $u+1$ if $u<k$; otherwise go to the next stage $s+2$.
\end{enumerate}
\item If not mentioned, set ${\tt state}_{s+1}(u)={\tt state}_s(u)$ and $\rho^{s+1}_u=\rho^s_u$.
\end{enumerate}

At the end of stage $s+1$, we will define $T_{s+1}$.
Consider the downward closure $T^\ast_{s+1}$ of the following set:
\[\{\rho_u^{s+1}\fr i:{\tt state}(u)=1\mbox{ and }i<2\}\cup\{\rho_u^{s+1}:{\tt state}(u)=2\}.\]

Let $T^{\ast,{\rm leaf}}_{s+1}$ be the set of all leaves of $T^\ast_{s+1}$.
Note that every element of $T^{\ast,{\rm leaf}}_{s+1}$ is extendible in $T_s$ since $\rho_u^{s+1}$ is extendible in $T_s$.
For each leaf $\rho\in T^{\ast,{\rm leaf}}_{s+1}$, if $|\rho|=s+1$ then put $\eta_\rho=\eta$; otherwise choose an extendible leaf $\eta\in T_s$ extending $\rho$, and define $\eta_\rho=\eta\fr 0$.

Let $T_0$ be an empty tree.
We define $T_{s+1}$ as follows:
\[T_{s+1}=T_s\cup\{\eta_\rho:\rho\in T^{\ast,{\rm leaf}}_{s+1}\}.\]

Note that the extendible nodes in $T_{s+1}$ are exactly the downward closure of $\{\eta_\rho:\rho\in T^{\ast,{\rm leaf}}_{s+1}\}$, and every element of $T_{s+1}^\ast$ is extendible in $T_{s+1}$, that is,
\begin{itemize}
\item If ${\tt state}_{s+1}(u)=1$, then $\rho^{s+1}_u\fr i$ is extendible in $T_{s+1}$ for each $i<2$.
\item If ${\tt state}_{s+1}(u)=2$, then $\rho^{s+1}_u$ is extendible in $T_{s+1}$.
\end{itemize}

Our definition of $(T_s)_{s\in\om}$ clearly satisfies the property (II) mentioned above.
Concerning the property (I), one can see the following:

\begin{lemma}\label{lem:tree-correct-instance}
$T_{s+1}$ has at least one, and at most $k+1$ extendible leaves.
\end{lemma}

\begin{proof}
The former assertion trivially holds since $\rho^s_0$ is always defined as an empty string for any $s\in\om^\om$.
For the latter assertion, it suffices to show that any branching extendible node of $T_{s+1}$ is of the form $\rho^{s+1}_u$ for some $u<k$.
This is because $T_s$ is binary, and then the above property automatically ensures that $T_s$ has at most $k+1$ extendible leaves.

Let $\sigma$ be a branching extendible node of $T_{s+1}$.
If $|\sigma|=s$, since $T_s$ is of height $s$, $\sigma$ is of the form $\rho^{s+1}_u$ by our definition of $T_{s+1}$ .
If $|\sigma|<s$, then it is also a branching extendible node of $T_s$ by the property (II) of our construction, and thus it is of the form $\rho^s_u$ by induction.
If $\rho^s_u=\rho^{s+1}_u$ for any $u$, then our Lemma clearly holds.
If $\rho^s_u\not=\rho^{s+1}_u$, then it can happen at (2a) or (3a), and thus, there is $v\leq u$ such that the $v$-th strategy has acted at stage $s+1$.
We claim that for any $\rho\in T_{s+1}^\ast$ we have $\rho^s_u\not\prec\rho$.
This claim implies that $\rho^s_u$ is not a branching extendible node in $T_{s+1}$, which is a contradiction, and therefore we must have $\rho^s_u=\rho^{s+1}_u$.

To show the claim, note that $\rho^{s+1}_w$ is undefined for $w>v$.
If $w<u$ and $\rho^s_w$ is defined then $\rho^s_u\not\preceq\rho^s_w$ by $\mathcal{S}_u$'s action at (2a).
If $w<v$ then $\rho^{s+1}_w=\rho^{s}_w$.
For $w=v$, if ${\tt state}_{s+1}(v)=1$ then $\mathcal{S}_v$ reaches at (2a) at stage $s+1$ and $\rho^s_v\not\preceq\rho^s_u$ by $\mathcal{S}_v$'s action.
If ${\tt state}_{s+1}(v)=2$ then $\mathcal{S}_v$ reaches at (3a) at stage $s+1$, and thus $\rho^{s+1}_v$ is a successor of $\rho^{s+1}_u$ and thus $\rho^s_u\not\prec\rho^{s+1}_v$.
Hence, there is no $\rho\in T_{s+1}^\ast$ such that $\rho^s_u\prec\rho$ as desired.
\end{proof}

\begin{lemma}\label{lem:sort-diagonalize}
If ${\tt state}_{s+1}(u)=2$, then $\Phi(\sort{\alpha\upto s+1}{u}\fr u^\om)$ is not extendible in $T_{s+1}$.
\end{lemma}

\begin{proof}
If ${\tt state}_{s+1}(u)=2$, then there is stage $t\leq s+1$ such that $\sort{\alpha\upto t}{u}=\sort{\alpha\upto s+1}{u}$ and the $u$-th strategy $\mathcal{S}_u$ arrives at (2a) at stage $s$ and (3a) at $s+1$, and the $u$-th strategy is not injured by any higher priority strategy during stages between $t$ and $s+1$, and in particular, $\rho^t_u=\rho^{s}_u$.
By our action (3a), $\Phi(\sort{\alpha\upto s+1}{u}\fr u^\om)$ extends the sister of $\rho^{s+1}_u$.
If $v>u$ then $\rho^{s+1}_v$ is undefined.
If $v<u$ and $\rho^{t}_v$ is undefined, then since no injury happens below $u$ during stages between $t$ and $s+1$, we have $\rho^{s}_u=\rho^{t}_u\not\preceq\rho^{t}_v=\rho^{s+1}_v$, which implies that $\rho^{s+1}_v$ does not extend the sister of $\rho^{s+1}_u$.
Hence the sister of $\rho^{s+1}_u$ does not extend to a leaf of $T^\ast_{s+1}$.
Therefore, $\Phi(\sort{\alpha\upto s+1}{u}\fr u^\om)$ is not extendible in $T_{s+1}$.
\end{proof}

We now verify our construction.
Put $T=\bigcup_{k}T_k$.
By Lemma \ref{lem:tree-correct-instance}, since our construction of $(T_s)_{s\in\om}$ satisfies the conditions (I) and (II), the set $[T]$ of all infinite paths through $T$ is an instance of ${\sf C}_{\#\leq k+1}$.
Let $\alpha$ be an instance of ${\sf Sort}_{k}$.

\begin{lemma}
$\Phi({\sf Sort}_{k}(\alpha))\not\in [T]$.
\end{lemma}

\begin{proof}
By pigeonhole principle, there exists $u$ such that $\alpha(i)=u$ for infinitely many $i$.
Let $u$ be the least such number.
Then there exists $s$ such that $\sort{\alpha}{}:=\sort{\alpha\upto s}{u}\fr u^\om$ is the right answer to the instance $\alpha$ of ${\sf Sort}_k$, that is, it is the result by sorting $\alpha$.
Then, for any $v\leq u$, the $v$-partial sort of $\alpha$ stabilizes after $s$, that is, $\sort{\alpha\upto t+1}{v}=\sort{\alpha\upto t}{v}$ for all $t\geq s$.
After the $v$-partial sort of $\alpha$ stabilizes, the $v$-th strategy $\mathcal{S}_v$ can injure lower priority strategies at most two times, i.e., at (2a) and (3a).
Therefore, there is stage $s_0\geq s$ such that the $u$-th strategy $\mathcal{S}_u$ is never injured by higher priority strategies after $s_0$.
Then, ${\tt state}_t(u)$ converges to some value.

\medskip
\noindent
{\bf Case 1.} $\lim_t{\tt state}_t(u)=0$.
By our choice of $s_0$, $\mathcal{S}_u$ always goes to (2b), and never goes to (2a) after $s_0$.
However, if $\Phi(\sort{\alpha}{})$ is an infinite string, then the strategy must go to (2a) since $\{\rho_v^s:v<u\}$ is finite.
Hence, $\Phi(\sort{\alpha}{})$ cannot be an infinite path through $T$.

\medskip
\noindent
{\bf Case 2.} $\lim_t{\tt state}_t(u)=1$.
Let $s_1\geq s_0$ be the least stage such that $\mathcal{S}_u$ reaches (2a) with some $\rho$.
We claim that if an extendible node in $T_t$ extends $\rho$, then it also extends $\rho^{t}_u$ for any $t>s_1$.
According to the condition of $\mathcal{S}_u$'s strategy (2), for any $v<u$, we have $\rho\not\preceq\rho^{s_1}_v=\rho^{s_0}_v$.
By injury in (2a), $\rho^{s_1}_v$ is undefined for any $v>u$.
Therefore, any extendible node of $T_{s_1+1}$ extends $\rho^t_v$ or $\rho^t_v\fr i$ for some $v\leq u$ and $i<2$.
Hence, if an extendible node in $T_{s_1+1}$ extends $\rho$, then it also extends $\rho^{s_1+1}_u=\rho^t_u$.
By the property (II) of our construction, the claim follows.
Now, by our assumption, $\mathcal{S}_u$ always goes to (3b), and never goes to (3a).
This means that $\Phi(\sort{\alpha\upto t}{u}\fr u^\om)$ extends $\rho$, but does not extend $\rho^{t}_u$ for any $t>s_1$.
Therefore, $\Phi(\sort{\alpha\upto t}{u}\fr u^\om)$ is not extendible in $T_t$ for any $t>s_1$.
Consequently, $\Phi(\sort{\alpha}{})\not\in[T]$.

\medskip
\noindent
{\bf Case 3.} $\lim_t{\tt state}_t(u)=2$.
Let $s_2\geq s_0$ be the least stage such that $\mathcal{S}_u$ reaches (3a).
Then by Lemma \ref{lem:sort-diagonalize}, $\Phi(\sort{\alpha\upto s_2}{u}\fr u^\om)$ is not extendible in $T_{s_2}$.
Since $\mathcal{S}_u$ is not injured after $s_0$, we conclude $\Phi(\sort{\alpha}{})\not\in[T]$.
\end{proof}

By Theorem \ref{theo:recursionweihrauch}, this implies the desired assertion.
\end{proof}

\section{The comparison game for products of finite choice}
\label{sec:game}
In this section we consider the question when finite choice for some cardinality is reducible to some finite product of finite choice operators. We do not obtain an explicit characterization, but rather an indirect one. We introduce a special reachability game (played on a finite graph), and show that the winner of this game tells us whether the reduction holds. This in particular gives us a decision procedure (which so far has not been implemented yet, though).

Our game is parameterized by numbers $k$, and $n_0,n_1,\ldots,n_\ell$. We call the elements of $\bigcup_{i \leq \ell} \{i\} \times n_i$ \emph{colours}, and the elements of $\Pi_{i \leq \ell} n_i$ \emph{tokens}. A token $w$ has colour $(i,c)$, if $w_i = c$.

The current board consists of up to $k$ boxes each of which contains some set of tokens, with no token appearing in distinct boxes. If there ever is an empty box, then Player 1 wins. If the game continues indefinitely without a box becoming empty, Player 2 wins. The initial configuration is chosen by Player 1 selecting the number of boxes, and by Player 2 distributing all tokens into these boxes.

The available actions are as follows:
\begin{description}
\item[Remove] Player 1 taps a box $b$. Player 2 selects some colours $C$ such that every token in $b$ has a colour from $C$. Then the box $b$ and all tokens with a colour from $C$ are removed.
\item[Reintroduce colour] Player 2 picks two `adjacent' colours $(i,c)$ and $(i,d)$, such that no token on the board has colour $(i,d)$. For every box $b$, and every token $w \in b$ having colour $c$, he then adds a token $w'$ to $b$ that is identical to $w$ except for having colour $(i,d)$ rather than $(i,c)$.
\item[Split box] If there are less than $k$ boxes on the board, Player 1 can select a box $b$ to be split into two boxes $b_0$ and $b_1$. Player 2 can chose how to distribute the tokens from $b$ between $b_0$ and $b_1$. Moreover, Player 2 can do any number of \emph{Reintroduce colour} moves before the \emph{Split box}-move takes effect.
\end{description}

\begin{theorem}
\label{theo:game}
$\C_{\sharp \leq k} \leqW \C_{\sharp \leq n_0} \times \ldots \C_{\sharp \leq n_\ell}$ iff Player 2 wins the comparison game for parameters $k$, $n_0, \ldots, n_\ell$.
\end{theorem}

The proof proceeds via Lemmas \ref{lemma:gameplayer2}, \ref{lemma:gameplayer1} below. We observe that the game is a reachability game played on a finite graph. In particular, it is decidable who wins the game for a given choice of parameters. An implementation of the decision procedure is in progress. We have only considered the case $n_i = 2$ so far, and know:

\begin{proposition}
\hfill
\begin{enumerate}
\item Player 2 wins for $k + 1 \leq \ell$.
\item Player 1 wins for $k + 1 \geq 2^{\ell-1}$
\end{enumerate}
\begin{proof}
The first claim follows from Theorem \ref{theo:game} in conjunction with \cite[Proposition 3.9]{paulyleroux} stating that $\ic{\leq n+1} \leqW \ic{\leq 2}^n$. The second is immediate when analyzing the game.
\end{proof}
\end{proposition}

\begin{lemma}
\label{lemma:gameplayer2}
From a winning strategy of Player 2 in the comparison game we can extract witnesses for the reduction $\C_{\sharp \leq k} \leqW \C_{\sharp \leq n_0} \times \ldots \C_{\sharp \leq n_\ell}$.
\begin{proof}
We recall that the input to $\C_{\sharp \leq k}$ can be seen as an infinite binary tree having at most $k$ vertices on each level. We view this tree as specifying a strategy for Player 1 in the comparison game: The boxes correspond to the paths existing up to the current level of the tree. If a path dies out, Player 1 taps the corresponding box. If a path splits into two, Player 1 splits the corresponding box.

Which tokens exist at a certain time tells us how the instances to $\C_{\sharp \leq n_0}, \ldots , \C_{\sharp \leq n_\ell}$ are built. The colour $(i,j)$ refers to the $j$-path through the $i$-th tree at the current approximation. If a colour gets removed, this means that the corresponding path dies out. If a colour gets reintroduced, we split the path corresponding to the duplicated colour into two.

It remains to see how the outer reduction witness maps infinite paths through these trees back to an infinite path through the input tree. If we are currently looking at some finite approximation of the input tree and the query trees, together with an infinite path through each query tree, then the infinite paths indicates some token which never will be removed. That means that any box containing that token never gets tapped, i.e.~that certain prefixes indeed can be continued to an infinite path.
\end{proof}
\end{lemma}

\begin{lemma}
\label{lemma:gameplayer1}
From a winning strategy of Player 1 in the comparison game we can extract a witness for the non-reduction $\C_{\sharp \leq k} \nleqW \C_{\sharp \leq n_0} \times \ldots \C_{\sharp \leq n_\ell}$ according to Theorem \ref{theo:recursionweihrauch}.
\begin{proof}
We need to describe a procedure that constructs an input for $\C_{\sharp \leq k}$ given inputs to $\C_{\sharp \leq n_0}, \ldots , \C_{\sharp \leq n_\ell}$ and an outer reduction witness. Inverting the procedure from Lemma \ref{lemma:gameplayer2}, we can view the given objects as describing a strategy of Player 2 in the game. We obtain the input tree to $\C_{\sharp \leq k}$ by observing how the winning strategy of Player 1 acts against this. When Player 1 taps the $i$-th box, we let the $i$-th path through the tree die out. When Player 1 splits the $i$-th box, we let both children of the $i$-th vertex present at the current layer be present at the subsequent layer. Otherwise, we keep the left-most child of any vertex on the previous layer.

Since Player 1 is winning, we will eventually reach an empty box. At that point, we let all other paths die out, and only keep the one corresponding to the empty box. This means that any path selected by the outer reduction witness we obtained Player 2's strategy from will fall outside the tree, and thus satisfy the criterion of Theorem \ref{theo:recursionweihrauch}.
\end{proof}
\end{lemma}

\section{Rectangles versus Triangles}
\label{sec:triangles}
In this section, we shall show that the product of one-dimensional convex choice with itself is strictly weaker than two-dimensional convex choice. We achieve this by comparing the strength of certain restrictions of two-dimensional convex choice. Let $\mathfrak{T}$ be the class of closed triangles in $[0,1]^2$. We consider the degenerate cases of lines and single points to be included. The main result of this section is:

\begin{theorem}
$\XC_2|_\mathfrak{T} \nleqW \XC_1 \times \XC_1$
\end{theorem}

Note that we can conceive of $\XC_1 \times \XC_1$ as choice for rectangles in $\uint^2$ with the restriction that the rectangles are aligned to the boundaries of the unit square. Glossing over the alignment-restriction, we could say that choice for triangles is not reducible to choice for rectangles. Again, the degenerate cases of lines and points would be included.

\begin{definition}
A class $\mathfrak{A} \subseteq \mathcal{A}(\mathbf{X})$ has \emph{triple activable sites}, if there are computable families $(A_n)_{n \in \mathbb{N}} \in \mathcal{A}(\mathbf{X})$, $(\overline{S}^0_n,\overline{S}^1_n,\overline{S}^2_n)_{n \in \mathbb{N}} \in \mathcal{K}(\mathbf{X})$ and $(S^0_n,S^1_n,S^2_n)_{n \in \mathbb{N}} \in \mathcal{O}(\mathbf{X})$ such that
\begin{enumerate}
\item Whenever $\forall n \ A_{p(n+1)} \subseteq A_{p(n)}$, then $\bigcap_{n \in \mathbb{N}} A_{p(n)} \in \mathfrak{A}$.
\item $\forall n \in \mathbb{N} \ \forall i \in \{0,1,2\} \ \emptyset \neq \overline{S}^i_n \subseteq S^i_n$
\item $\forall n \in \mathbb{N} \ \forall i,j \in \{0,1,2,\} \ i \neq j \Rightarrow S^i_n \cap S^j_n = \emptyset$
\item $\forall n \in \mathbb{N} \ \forall i \in \{0,1,2\} \overline{S}^i_n \subseteq A_n$
\item There is a computable procedure (choosing a site) that given $n \in \mathbb{N}$ and $i \in \{0,1,2\}$ computes $m \in \mathbb{N}$ such that $A_m \subseteq \overline{S}_n^i$.
\item There is a computable procedure (killing a site) that given $n \in \mathbb{N}$ and $i \in \{0,1,2\}$ computes $m \in \mathbb{N}$ such that $A_m \subseteq A_n \setminus S_n^i$.

\item There is a computable procedure (activating one site) that given $n \in \mathbb{N}$, $i \in \{0,1,2\}$ computes $m_{0},m_{1}$ such that
\begin{enumerate}
\item $\forall \ell \in \{0,1,2\} \ \forall k \in \{0,1\} \ A_{m_k} \subseteq A_n \wedge S_{m_k}^\ell \subseteq S_n^\ell \wedge \overline{S}_{m_k}^\ell \subseteq \overline{S}_n^\ell$
\item $\forall a \in \{0,1\} \ S_{m_{a}}^i \cap A_{m_{(1-a)}} = \emptyset$
\end{enumerate}

\item There is a computable procedure (activating two sites) that given $n \in \mathbb{N}$, $i,j \in \{0,1,2\}$, $i \neq j$ computes $m_{00},m_{01},m_{10},m_{11}$ such that
\begin{enumerate}
\item $\forall \ell \in \{0,1,2\} \ \forall k \in \{00,01,10,11\} \ A_{m_k} \subseteq A_n \wedge S_{m_k}^\ell \subseteq S_n^\ell \wedge \overline{S}_{m_k}^\ell \subseteq \overline{S}_n^\ell$
\item $\forall a, b \in \{0,1\} \ S_{m_{ab}}^i \cap A_{m_{(1-a)b}} = \emptyset \wedge  S_{m_{ab}}^j \cap A_{m_{a(1-b)}}=\emptyset$
\end{enumerate}

\item There is a computable procedure (activating two sites) that given $n \in \mathbb{N}$, computes $(m_{w})_{w \in \{0,1\}^3}$ such that
\begin{enumerate}
\item $\forall \ell \in \{0,1,2\} \ \forall k \in \{0,1\}^3 \ A_{m_k} \subseteq A_n \wedge S_{m_k}^\ell \subseteq S_n^\ell \wedge \overline{S}_{m_k}^\ell \subseteq \overline{S}_n^\ell$
\item $\forall a, b, c \in \{0,1\} \ S_{m_{abc}}^0 \cap A_{m_{(1-a)bc}} = \emptyset \wedge  S_{m_{abc}}^1 \cap A_{m_{a(1-b)c}} =\emptyset\wedge  S_{m_{abc}}^2 \cap A_{m_{ab(1-c)}}=\emptyset$
\end{enumerate}
\end{enumerate}
\end{definition}

The idea is that we can construct sets of type $\mathfrak{A}$ by selecting three disjoint regions inside them (the sites), in a way that we always later entirely remove one of those sites without leaving $\mathfrak{A}$. Moreover, we can subdivide two sites simultaneously, and then later on move to a subset belonging to $\mathfrak{A}$ that realizes any combination of the subdivisions.

\begin{lemma}
$\mathfrak{T}$ has triple activable sites.
\begin{proof}
We chose as sites the tips of the triangle. For subdivision, we split the chosen tips by the corresponding median. We can find suitable smaller triangles avoiding half of a each tip but intersecting the other as depicted in Figure \ref{fig:triangles}.

\begin{figure}[h]
\centering
\includegraphics[width=0.9\textwidth]{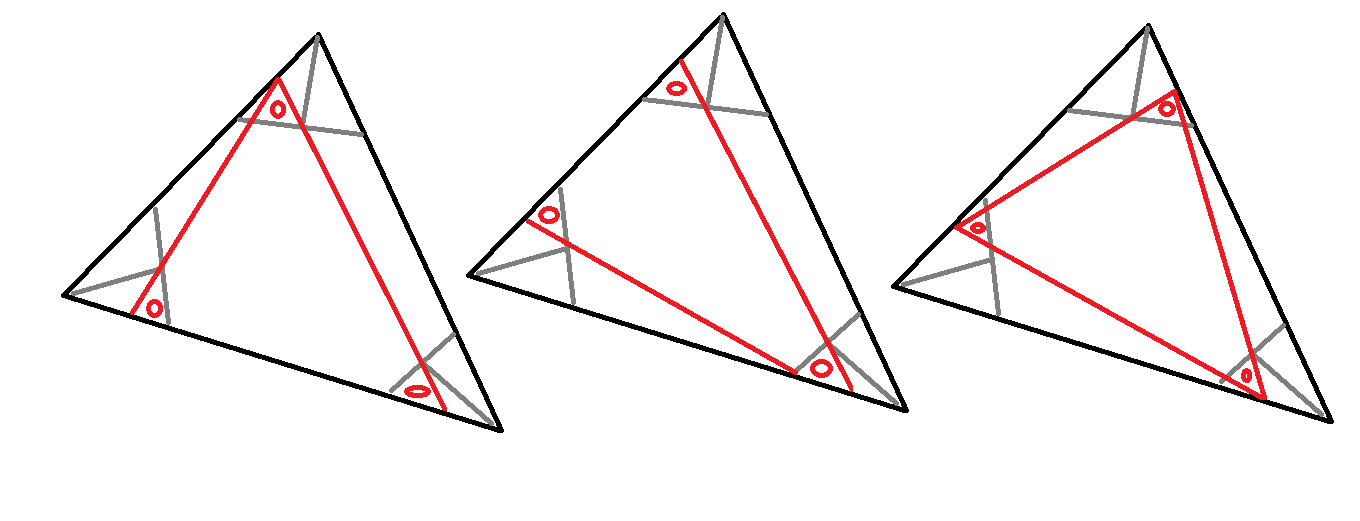}
\caption{Activating sites in a triangle}
\label{fig:triangles}
\end{figure}
\end{proof}
\end{lemma}

\begin{lemma}
\label{lemma:boxesargument}
Let $\mathfrak{A} \subseteq \mathcal{A}(\mathbf{X})$ have triple activable sites. Then $\C_\mathbf{X}|_\mathfrak{A} \nleqW \XC_1 \times \XC_1$.
\begin{proof}
Assume for the sake of a contradiction that $\C_\mathbf{X}|_\mathfrak{A} \leqW \XC_1 \times \XC_1$ via $K$, $H$. We apply the reduction to input $A_0$. Then $K(A_0,\cdot) : H(A_0) \to A_0$ is a continuous function. By uniform continuity of that function, we can obtain a prefix $w$ of $A_0$ and rectangular grid covering $H(A_0)$ such that for every cell and for every $i \in \{0,1,2\}$ it holds that $K(w,\cdot)$ maps the cell into $S_0^i$ or its image is disjoint from $\overline{S}_0^i$, without the function needing to know more than $w$ about its first input. We assign colours $\{0,1,2,\bot\}$ to cells in such a way that if a cell has colour $i \in \{0,1,2\}$, then it is mapped into $S_0^i$, and if its image intersects $\overline{S}_0^i$, then it has colour $i$. We call a cell \emph{coloured}, if its colour is $\{0,1,2,\}$, it is uncoloured otherwise.

We will in the process of the construction move from $A_0$ as input to other $A_n$ as chosen by the \emph{activating two sites} procedure. When doing that, we adjust our colour assignment to work with the $S_n^i$ and $\overline{S}_n^i$ instead. We observe that the only change in colour this can cause is to recolour a cell from $i \in \{0,1,2\}$ to $\bot$. Moreover, cells can disappear by falling outside of $H(A_)n)$. Note that this means that the number of coloured cells is thus a non-increasing property with a finite initial value. Should the situation arise where some colour $i \in \{0,1,2\}$ is not present as a colour of a cell at all, then choosing the corresponding site and moving the input $A_m \subseteq \overline{S}_n^i$ breaks the reduction, and provides the desired contradiction.

Generally, we will pick two cells $C_1,C_2$ and activate the sites corresponding to their colours. We inspect the outer reduction witness further to see which parts of $C_1$ are mapped into which of $S_{m_{ab}}^i$ and which parts of $C_2$ into which of $S_{m_{ab}}^j$. If any of those is not present in of the two cells, then moving on to the corresponding $A_{m_ab}$ as input next would force that cell to lose its colour. Assume all choices are present in both cells. Moving into $A_{m_{ab}}$ means that we need to remove $S^i_{m_{(1-a)b}}$ and $S^j_{m_{a(1-b)}}$ from $H(A_{a_{ab}})$, in particular from the cells $C_1$ and $C_2$. Moreover, $H(A_{a_{ab}})$ has to be rectangular and aligned to the boundary. We can thus view activating two sites as being followed by a choice of subregions within the corresponding cells (chosen by the opponent), and there subsequently being vertical or horizontal cuts that remove the subregions not selected by us.

We proceed to argue that in most configurations, these cuts will decrease the number of coloured cells by removing some of them. We make a case distinction based on the arrangement of coloured cells.



{\bf Case 1:} There is a coloured cell $C_0$ such that other coloured cells are in at least three cardinal directions

We activate the colour corresponding to $C_0$. As above, this corresponds to either a horizontal or a vertical cut. By assumption, at least one side of the cut is such that moving into it removes another coloured cell. This is illustrated in Figure \ref{fig:boxescase2}.

\begin{figure}[h]
\centering
\includegraphics[width=0.3\textwidth]{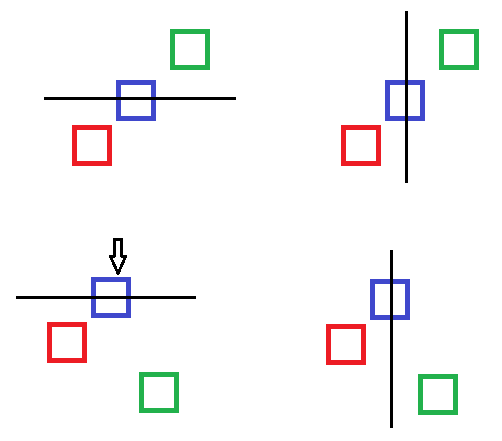}
\caption{The subcases of Case 1 in the proof of Lemma \ref{lemma:boxesargument}}
\label{fig:boxescase2}
\end{figure}

{\bf Case 2:} There are two different-coloured cells aligned horizontally or vertically, and another coloured cell off that alignment.

If Case 1 does not apply, we have (up to symmetry) the configuration depicted in Figure \ref{fig:boxescase3}. We can only cut away from the red cell from west and north, and from the blue cell only from north and east -- any other cut would remove an entire cell. We activate the two sites corresponding to red and blue, and consider the subregions in the two cells which we would remove by a cut from the north. We choose one that extends furthest south, and chose to remove it, but not the other. This is impossible with the allowed cuts.

\begin{figure}[h]
\centering
\includegraphics[width=0.3\textwidth]{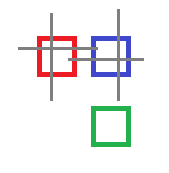}
\caption{Case 2 in the proof of Lemma \ref{lemma:boxesargument}}
\label{fig:boxescase3}
\end{figure}

{\bf Case 3:} There are three different-coloured cells aligned horizontally or vertically.

Up to symmetry, the configuration is depicted in Figure \ref{fig:boxescase4}. We activate all three sites. In the red box, there is at most one subregion which we can remove by cutting from the west, the other one (called $r$) requires a north or south cut. Symmetrically, in the green box, there is also one subregion (called $g$) which requires a north or south cut. W.l.o.g.~,assume that of these two subregions, the red one $r$ extends further south. In the blue box, the two subregions need to be removed by cuts from north or south (call the northern one $n$ and the southern one $s$). Now it is impossible to remove the subregions $r$ and $n$ but not $g$ with cuts that do not remove entire coloured cells.
\begin{figure}[h]
\centering
\includegraphics[width=0.5\textwidth]{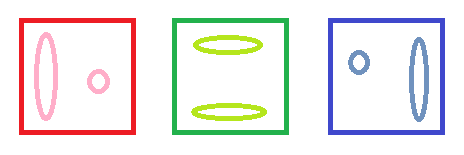}
\caption{Case 3 in the proof of Lemma \ref{lemma:boxesargument}}
\label{fig:boxescase4}
\end{figure}

\end{proof}
\end{lemma}

\section{Convex choice and compositions}
\label{sec:convexcomposition}

Our next theorem gives a negative answer to \cite[Question 3.14]{paulyleroux}, which was again raised as an open problem at a Dagstuhl seminar on Weihrauch reducibility, c.f.~\cite{pauly-dagstuhl}:

\begin{theorem}\label{maintheorem2}
${\sf XC}_1\star{\sf AoUC}\not\leq_W{\sf XC}_k$ for all $k\in\mathbb{N}$.
\end{theorem}

\begin{proof}
We will need some geometric arguments involving convexity, measure and dimension.
If a convex set $X\subseteq[0,1]^k$ is at most $d$-dimensional, then $X$ is included in a $d$-dimensional hyperplane $L\subseteq[0,1]^k$ by convexity.
It is easy to define the $d$-dimensional Lebesgue measure $\lambda^d$ on $L$ which is consistent with the $d$-dimensional volume on $d$-parallelotopes in $[0,1]^k$.

Let $(X[s])_{s\in\omega}$ be an upper approximation  of a convex closed set $X\subseteq[0,1]^k$.
Even if we know that $X$ is at most $d$-dimensional for some $d<k$, it is still possible that $X[s]$ can always be at least $k$-dimensional for all $s\in\omega$.
Fortunately, however, by compactness one can ensure that for such $X$, say $X\subseteq L$ for some $d$-hyperplane $L$ by convexity, $X[s]$ for sufficiently large $s$ is eventually covered by a {\em thin} $k$-parallelotope $\widehat{L}$ obtained by expanding $d$-hyperplane $L$.
For instance, if $X\subseteq[0,1]^3$ is included in the plane $L=\{1/2\}\times[0,1]^2$, then for all $t\in\omega$, there is $s\in\omega$ such that $X[s]\subseteq\widehat{L}(2^{-t}):=[1/2-2^{-t},1/2+2^{-t}]\times[0,1]^2$ by compactness.
We call such $\widehat{L}(2^{-t})$ as the {\em $2^{-t}$-thin expansion} of $L$.

We give a formal definition of the $2^{-t}$-thin expansion of a subset $Y$ of a hyperplane $L$.
A $d$-hyperplane $L\subseteq[0,1]^k$ is named by a $d$-many linearly independent points in $[0,1]^k$.
If $(x_i)_{i<d}$ is linearly independent, then this fact is witnessed at some finite stage.
Therefore, there is a computable enumeration $(L^d_e)_{e\in\omega}$ of all rational $d$-hyperplanes.
A rational closed subset of $L$ is the complement of the union of finitely many rational open balls in $L$.
Given a pair $(L,Y)$ of (an index of) a rational $d$-hyperplane $L\subseteq[0,1]^k$ and a rational closed set $Y\subseteq L$, we define the {\em $2^{-t}$-thin expansion of $Y$ on $L$} as follows:
We calculate an orthonormal basis $(\mathbf{e}_1,\dots,\mathbf{e}_{k-d})$ of the orthogonal complement of the vector space spanned by $L-v$ where $v\in L$, and define
\[\widehat{Y}(2^{-t})=\left\{v+\sum_{i=1}^{k-d} a_i\mathbf{e}_i:v\in Y\mbox{ and }-2^{-t}\leq a_i\leq 2^{-t}\mbox{ for any }i\leq k-d\right\},\]

Since $Y$ is a rational closed set, we can compute the measure $\lambda^d(Y)$.
Indeed, we can compute the maximum value $m^d(Y,t)$ of $\lambda^d(\widehat{Y}(2^{-t})\cap L')$ where $L'$ ranges over all $d$-dimensional hyperplanes.
For instance, if $Y=[0,s]\times\{y\}$, it is easy to see that $m^1(Y,2^{-t})$ is the length $\sqrt{s^2+2^{-2t+2}}$ of the diagonal of the rectangle $\widehat{Y}(2^{-t})=[0,s]\times[y-2^{-t},y+2^{-t}]$.

Let us assume that we know that a convex set $X\subseteq[0,1]^k$ is at most $d$-dimensional, and moreover, a co-c.e.~closed subset $\tilde{X}$ of $X$ satisfies that $\lambda^d(\tilde{X})<r$.
What we will need is to find, given $\varepsilon>0$, stage $s$ witnessing that $\lambda^d(\tilde{X})<r+\varepsilon$ under this assumption.
Of course, generally, there is no stage $s$ such that $\lambda^d(\tilde{X}[s])<r+\varepsilon$ holds since $\tilde{X}[s]$ may not even be $d$-dimensional for all $s\in\omega$ as discussed above.
To overcome this obstacle, we again consider a thin expansion of (a subset of) a hyperplane.
Indeed, given $t>0$, there must be a rational closed subset $Y$ of a $d$-dimensional rational hyperplane $L$ such that $Y$ is very close to $\tilde{X}$, and that $\tilde{X}$ is covered by the $2^{-t}$-thin expansion $\widehat{Y}(2^{-t})$ of $Y$.
That is, by compactness, it is not hard to see that given $\varepsilon>0$, one can effectively find $s,Y,t$ such that
\[\tilde{X}[s]\subseteq\widehat{Y}(2^{-t}) \mbox{ and } m^d(Y,t)<r+\varepsilon.\]

In this way, if the inequality $\lambda^d(\tilde{X})<r$ holds for a co-c.e.~closed subset $\tilde{X}$ of a $d$-dimensional convex set $X$, then one can effectively confirm this fact.

We next consider some nice property of an admissible representation of $[0,1]^k$.
It is well-known that $[0,1]^k$ has an admissible representation $\delta$ with an effectively compact domain $[T_\delta]$ such that $\delta$ is an effectively open map (see \cite{weihrauchd}).
In particular, $\delta^{-1}[P]$ is compact for any compact subset $P\subseteq[0,1]^k$.
Additionally, we can choose $\delta$ so that, given a finite subtree $V\subseteq T_\delta$, one can effectively find an index of the co-c.e.~closed set ${\rm cl}(\delta[V])$, and moreover, $\lambda^{d}(L\cap {\rm cl}(\delta[V])\setminus\delta[V])=0$ for any $d$-dimensional hyperplane $L\subseteq[0,1]^k$ for $d\leq k$.
For instance, consider a sequence of $2^{-n}$-covers $(\mathcal{C}^n_k)_{k<b(n)}$ of $[0,1]^k$ consisting of rational open balls, and then each $b$-bounded string $\sigma$ codes a sequence of open balls $(\mathcal{C}^n_{\sigma(n)})_{n<|\sigma|}$.
Then we may define $T_\delta$ as the tree consisting of all $b$-bounded sequences that code strictly shrinking sequences of open balls, and $\delta(p)$ as a unique point in the intersection of the sequence coded by $p\in[T_\delta]$.
It is not hard to verify that $\delta$ has the above mentioned properties.
Hereafter, we fix such a representation $\delta:[T_\delta]\to[0,1]^k$.

Now we are ready to prove the assertion.
Let ${\rm ITV}_{[0,1]}$ denote the subspace of $\mathcal{A}([0,1])$ consisting of nonempty closed intervals in $[0,1]$.
Consider the following two partial multi-valued functions:
\begin{align*}
Z_0&:={\sf AoUC}\times{\rm id}:{\rm dom}({\sf AoUC})\times C(2^\mathbb{N},{\rm ITV}_{[0,1]})\rightrightarrows 2^\mathbb{N}\times C(2^\mathbb{N},{\rm ITV}_{[0,1]}),\\
Z_0&(T,J)={\sf AoUC}(T)\times\{J\},\\
Z_1&:=({\rm id}\circ\pi_0,{\sf XC}_1\circ{\rm eval}):2^\mathbb{N}\times C(2^\mathbb{N},{\rm ITV}_{[0,1]})\rightrightarrows 2^\mathbb{N}\times [0,1],\\
Z_1&(x,J)=\{x\}\times{\sf XC}_1(J(x)).
\end{align*}

Clearly, $Z_0\leq_W{\sf AoUC}$ and $Z_1\leq_W{\sf XC}_1$.
We will show that $Z_1\circ Z_0\not\leq_{\rm W}{\sf XC}_k$.

We will apply Theorem \ref{theo:recursionweihrauch}.
Let $\{(P_e,\varphi_e,\psi_e)\}_{e\in\mathbb{N}}$ be an effective enumeration of all co-c.e.~closed subsets of $[0,1]^k$, partial computable functions $\varphi_e:\subseteq \mathbb{N}^\mathbb{N}\to 2^\mathbb{N}$ and $\psi_e:\subseteq\mathbb{N}^\mathbb{N}\to[0,1]$.
Intuitively, $(P_e,\varphi_e,\psi_e)$ is a triple constructed by the opponent ${\sf Opp}$, who tries to show $Z_1\circ Z_0\leq_W{\sf XC}_k$ for some $k$.
The game proceeds as follows:
We first give an instance $(T_r,J_r)$ of $Z_1\circ Z_0$.
Then, ${\sf Opp}$ reacts with an instance $P_r$ of ${\sf XC}_k$, that is, a convex set $P_r\subseteq[0,1]^k$, and ensure that whenever $z$ is a name of a solution of $P_r$, $\varphi_r(z)=x$ is a path through $T_r$ and $\psi_r(z)$ chooses an element of the interval $J_r(x)$, where {\sf Opp} can use information on (names of) $T_r$ and $J_r$ to construct $\varphi_r$ and $\psi_r$.
By Theorem \ref{theo:recursionweihrauch}, to show the desired assertion, we only need to prevent ${\sf Opp}$'s strategy.

Hereafter, $P_e[s]$ denotes the stage $s$ upper approximation of $P_e$.
We identify a computable function $\varphi_e$ ($\psi_e$, resp.)~with a c.e.~collection $\Phi_e$ of pairs $(\sigma,\tau)$ of strings $\sigma\in\mathbb{N}^{<\mathbb{N}}$ and $\tau\in 2^{<\mathbb{N}}$ ($\Psi_e$ of pairs $(\sigma,D)$ of strings $\sigma\in\mathbb{N}^{<\mathbb{N}}$ and rational open intervals in $[0,1]$, resp.)~indicating that $\varphi_e(x)\succ\tau$ for all $x\succ\sigma$ ($\psi_e(x)\in D$ for all $x\succ\sigma$, resp.)
We use the following notations:
\begin{align*}
\Phi_e^q[s]&=\{(\sigma,\tau)\in\Phi_e[s]:|\tau|\geq q\},\\
\Psi_e^t[s]&=\{(\sigma,D)\in\Psi_e[s]:{\rm diam}(D)<3^{-t}\}.
\end{align*}

For a relation $\Theta\subseteq X\times Y$, we write ${\rm Dom}\Theta$ for the set $\{x\in X:(\exists y\in Y)\;(x,y)\in\Theta\}$.
We also use the following notations:
\begin{align*}
(\Phi_e[s])^{-1}[\rho]&=\bigcup\{\sigma\in{\rm Dom}\Phi_e^{|\rho|}[s]:\tau\succeq\rho\},\\
(\Psi_e[s])^{-1}_t[I]&=\bigcup\{\sigma:(\exists D)\;(\sigma,D)\in\Psi^{t}_e[s],\;{\rm diam}(D)<3^{-t}\mbox{ and }D\cap I\not=\emptyset\},
\end{align*}

Given $e$, we will construct a computable a.o.u.~tree $T_e$ and a computable function $J_e:2^\mathbb{N}\to{\rm ITV}_{[0,1]}$ in a computable way uniformly in $e$.
These will prevent ${\sf Opp}$'s strategy, that is, there is a name $z$ of a solution of $P_e$ such that if $\varphi_e(z)=x$ chooses a path through $T_e$ then $\psi_e(z)$ cannot be an element of the interval $J_e(x)$.
We will also define ${\tt state}(e,q,s)\in\mathbb{N}\cup\{\tt end\}$.
The value ${\tt state}(e,q,s)=t(q)$ indicates that at stage $s$, the $q$-th substrategy of the $e$-th strategy executes the action {\em forcing the measure $\lambda^{k-q}(\tilde{P}_e)$ of a nonempty open subset $\tilde{P}_e$ of $P_e$ to be less than or equal to $2^{q-t(q)}\cdot\varepsilon_{t(q)}$} with $\varepsilon_{t(q)}:=\sum_{j=0}^{t(q)+1}2^{-j}<2$.
Therefore, if ${\tt state}(e,q,s)$ tends to infinity as $s\to\infty$, then the $q$-th substrategy eventually {\em forces $P_e$ to be at most $(k-q-1)$-dimensional} under the assumption that $P_e$ is convex.
First define ${\tt state}(e,q,0)=0$, and we declare that the $q$-th substrategy is sleeping (i.e., not active) at the beginning of stage $0$.
At stage $s$, we inductively assume that $T_e\cap 2^{s-1}$ and ${\tt state}(e,q,s-1)$ have already been defined, say ${\tt state}(e,q,s-1)=t(q)$, and if ${\tt state}(e,q,0)\not={\tt end}$ then $T_e\cap 2^{s-1}=2^{s-1}$.

Our strategy is as follows:
\begin{enumerate}
\item At the beginning of stage $s$, we start to monitor the first substrategy $p$ which is still asleep.
That is, we calculate the least $p<k$ such that ${\tt state}(e,p,s-1)=0$.
If there is no such $p\leq k$, then go to (3).
Otherwise, go to (2) with such $p\leq k$.
\item Ask whether $\varphi_e(z)$ already computes a node of length at least $p+1$ for any name $z$ of an element of $P_e$.
In other words, ask whether $\delta^{-1}[P_e]\subseteq[{\rm Dom}\Phi^{p+1}_e[s]]$ is witnessed by stage $s$.
By compactness, if this inclusion holds, then it holds at some stage.
\begin{enumerate}
\item If no, go to substage $0$ in the item (4) after setting ${\tt state}(e,p,s)=0$.
\item If yes, the substrategy $p$ now starts to act (we declare that the substrategy $q$ is {\em active} at any stage after $s$), and go to (3).
\end{enumerate}
\item For each substrategy $q$ which is active at stage $s$, ask whether there is some $\tau\in 2^{q+1}$ such that any point in $P_e$ has a name $z$ such that $\varphi_e(z)$ does not extend $\tau$.
In other words, ask whether
\[(\exists\tau\in 2^{q+1})\;P_e\subseteq\bigcup\{\delta[\varphi_e^{-1}[\rho]]:\rho\in 2^{q+1}\mbox{ and }\rho\not=\tau\}\]
is witnessed by stage $s$.
Note that $\delta[\varphi_e^{-1}[\rho]]$ is c.e.~open since it is the image of a c.e.~open set under an effective open map.
Therefore, by compactness, if the above inclusion holds, then it holds at some stage.
\begin{enumerate}
\item If no for all such $q$, go to substage $0$ in the item (4).
\item If yes with some $q$ and $\tau$, we finish the construction by setting ${\tt state}(e,0,s)={\tt end}$ after defining $T_e$ as a tree having a unique infinite path $\tau\fr 0^\omega$.
This construction witnesses that any point of $P_e$ has a name $z$ such that $\varphi_e(z)\not\in[T_e]$ and hence, ${\sf Opp}$'s strategy fails.
\end{enumerate}
\item Now we describe our action at substage $q$ of stage $s$.
If $q\geq k$ or $q$ is not active at stage $s$, go to (1) at the next stage $s+1$ after setting $T_e\cap 2^{s}=2^s$.
Otherwise, go to (5).
\item Ask whether for any name $z$ of a point of $P_e$, whenever $\varphi_e(z)$ extends $0^q1$, the value of $\psi_e(z)$ is already approximated with precision $3^{-t(q)-2}$.
In other words, ask whether
\[\delta^{-1}[P_e]\cap\varphi_e^{-1}[0^q1]\subseteq[{\rm Dom}\Psi_e^{t(q)+2}[s]]\]
is witnessed by stage $s$.
Again, by compactness, this is witnessed at some finite stage.
\begin{enumerate}
\item If no, go to substage $q+1$ after setting ${\tt state}(e,q,s)=t(q)$.
\item If yes, go to (6).
\end{enumerate}
\end{enumerate}

Before describing the action (6), we need to prepare several notations.
We first note that $\delta^{-1}[P_e]\cap\varphi_e^{-1}[0^q1]$ is compact, and therefore, there is a tree $V^q\subseteq T_\delta$ (where ${\rm dom}(\delta)=[T_\delta]$) such that $[V^q]=\delta^{-1}[P_e]\cap\varphi_e^{-1}[0^q1]$.
Moreover, since we answered in the affirmative in the item (5), by compactness, there is a sufficiently large height $l$ such that every $\sigma\in V^q$ of length $l$ has an initial segment $\sigma'\preceq\sigma$ such that $(\sigma',D_\sigma)\in\Psi^{t(q)+2}_e$ for some interval $D_\sigma\subseteq[0,1]$ with ${\rm diam}(D_\sigma)<3^{-t(q)-2}$.
Given $\sigma\in V^q$ of length $l$, one can effectively choose such $D_\sigma$.
We will define pairwise disjoint intervals $I_0$ and $I_1$ which are sufficiently separated so that if ${\rm diam}(D)<3^{-t(q)-2}$ then $D$ can only intersects with one of them.
Then for every $\sigma\in V^q$ of length $l$, we define $h_{t(q)}(\sigma)=i$ if $D_\sigma\cap I_i\not=\emptyset$ for some $i<2$, otherwise put $h_{t(q)}(\sigma)=2$.

Now we inductively assume that $J_e(0^q1)[s-1]$ is a closed interval of the form $[3^{-t(q)}\cdot k,3^{-t(q)}\cdot (k+1)]$ for some $k\in\mathbb{N}$. Then, define $I_i=[3^{-t(q)-1}\cdot (3k+2i),3^{-t(q)-1}\cdot (3k+2i+1)]$ for each $i<2$.
Note that $I_0$ and $I_1$ be pairwise disjoint closed subintervals of $J_e(0^q1)[s-1]$.
Moreover, $I_0$ and $I_1$ satisfy the above mentioned property since the distance between $I_0$ and $I_1$ is $3^{-t(q)-1}$.
Therefore, $h$ is well-defined on $V^q\cap \omega^l$.

We consider ${\rm cl}(\delta[V^q])=P_e\cap{\rm cl}(\delta[\varphi^{-1}[0^q1]])$.
By the property of $\delta$, the set $P^q_e$ is co-c.e.~closed, and $\lambda^{k-q}({\rm cl}(\delta[V^q])\setminus\delta[V^q])=0$ whenever $P_e$ is at most $(k-q)$-dimensional.
Then, define $Q^{t(q)}_i$ for each $i<2$ as the set of all points in ${\rm cl}(\delta[V^q])$ all of whose names are still possible to have $\psi_e$-values in $I_i$.
More formally, define $Q^{t(q)}_i$ as follows:
\begin{align*}
V^{t(q)}_i&=V^q\cap 2^l\cap\left(h_{t(q)}^{-1}\{1-i\}\cup h_{t(q)}^{-1}\{2\}\right),\\
Q^{t(q)}_i&={\rm cl}(\delta[V^q])\setminus\delta[V^{t(q)}_i].
\end{align*}

Obviously, $V^{t(q)}_0\cup V^{t(q)}_1=V^q\cap 2^l$, and $Q^{t(q)}_i$ is effectively compact since $V^{t(q)}_i$ generates a clopen set for each $i<2$.
Moreover, we have that $\lambda^{k-q}(Q^{t(q)}_0\cap Q^{t(q)}_1)=0$ whenever $P_e$ is at most $(k-q)$-dimensional since $Q^{t(q)}_0\cap Q^{t(q)}_1\subseteq{\rm cl}(\delta[V^q])\setminus\delta[V^q]$ and $\lambda^{k-q}({\rm cl}(\delta[V^q])\setminus\delta[V^q])=0$.
Now, the $q$-th substrategy believes that we have already forced $\lambda^{k-q}(\delta[V^q])\leq 2^{q-t(q)+1}\cdot\varepsilon_{t(q)-1}$ and therefore, $\lambda^{k-q}(Q^{t(q)}_i)\leq 2^{q-t(q)}\cdot\varepsilon_{t(q)-1}$ for some $i<2$ since $\lambda^{k-q}(Q^{t(q)}_0\cap Q^{t(q)}_1)=0$ as mentioned above.
Here recall that we have $1\leq\varepsilon_{t(q)-1}<\varepsilon_{t(q)}<2$.
Now, we state the action (6):

\begin{enumerate}
\item[(6)]
Ask whether by stage $s$ one can find a witness for the condition that $\lambda^{k-q}(Q^{t(q)}_i)\leq 2^{q-t(q)}\cdot\varepsilon_{t(q)-1}$ for some $i<2$.
That is, ask whether one can find $Y,t,i$ by stage $s$ such that
\[Q^{t(q)}_i[s]\subseteq\widehat{Y}(2^{-t})\mbox{ and }m^{k-q}(Y,t)<2^{q-t(q)}\cdot\varepsilon_{t(q)}.\]
\begin{enumerate}
\item If no, go to substage $q+1$ after setting ${\tt state}(e,q,s)=t(q)$.
\item If yes, define $J_e(0^q1)[s]=I_i$, and go to substage $q+1$ after setting ${\tt state}(e,q,s)=t(q)+1$.
\end{enumerate}
\end{enumerate}

Eventually, $T_e$ is constructed as an a.o.u.~tree, and $J_e(x)$ is an nonempty interval for any $x$.
We now show that if {\sf Opp}'s reaction to our instance $(T_e,J_e)$ is $(P_e,\varphi_e,\psi_e)$, then {\sf Opp} loses the game.

\begin{claim}
Suppose that $P_e$ is a nonempty convex subset of $[0,1]^k$.
Then, there is a realizer $G$ of ${\sf XC}_k$ such that $(\varphi_e\circ G(\delta^{-1}[P_e]),\psi_e\circ G(\delta^{-1}[P_e])$ is not a solution to $Z_1\circ Z_0(T_e,J_e)$, that is, $\varphi_e\circ G(\delta^{-1}[P_e])\not\in[T_e]$ or otherwise $\psi_e\circ G(\delta^{-1}[P_e])\not\in J_e\circ\varphi_e\circ G(\delta^{-1}[P_e])$.
\end{claim}

\begin{proof}
Suppose not, that is, {\sf Opp} wins the game with $(P_e,\varphi_e,\psi_e)$ as a witness.
We first assume that $P_e$ is at most $(k-q)$-dimensional.
By our assumption, $\varphi_e$ is defined on all points in $\delta^{-1}[P_e]$.
By compactness of $\delta^{-1}[P_e]$, there is stage $s$ satisfying (2).

Suppose that there is $\tau\in 2^{q+1}$ such that $P_e[s]\subseteq\bigcup\{\delta[\varphi_e^{-1}[\rho]]:\rho\in 2^{p+1}\mbox{ and }\rho\not=\tau\}$.
This means that any point $x\in P_e[s]$ has a name $\alpha_x\in\delta^{-1}\{x\}$ such that $\varphi_e(\alpha_x)\not\succ\tau$.
Since $P_e\not=\emptyset$, if a realizer $G$ chooses such a name $\alpha_x$ of a point $x\in P_e$, then $\varphi_e\circ G(\delta^{-1}[P_e])\not\in[T_e]=\{\tau\fr 0^\omega\}$, that is, {\sf Opp} fails to find a path of $T_e$, which contradicts our assumption.

Thus, $\delta^{-1}[P_e]\cap\varphi_e^{-1}[0^q1]$ is nonempty.
Since this is compact and $\psi_e$, for any $t$, there is stage $s$ satisfying (5) with $t(q)=t$.
We can always assume that $\lambda^{k-q}(P_e)<2^{q+1}$ since $P_e$ is an at most $(k-q)$-dimensional convex set, $P_e$ is included in a $(k-q)$-dimensional hyperplane, and the $(k-q)$-dimensional Lebesgue measure of any $(k-q)$-dimensional hyperplane in $[0,1]^k$ is at most $\sqrt{q+1}<2^{q+1}$.
Thus, if $t(q)=0$ then $\lambda(\delta[V^q])\leq 2^{q-t(q)+1}\cdot\varepsilon_{t(q)-1}$ holds since $\delta[V^q]\subseteq P_e$ and $\varepsilon_{-1}=1$.

If $\lambda(\delta[V^q])\leq 2^{q-t(q)+1}\cdot\varepsilon_{t(q)-1}$, then $\lambda(Q^{t(q)}_i)\leq 2^{q-t(q)}\cdot\varepsilon_{t(q)-1}$ for some $i<2$ as discussed above.
Therefore, by the argument discussed above, at some stage $s$, one can find a rational closed subset $Y$ of a $(k-q)$-dimensional hyperplane, $t\in\mathbb{N}$, and $i<2$ satisfying the condition in the item (6).
At this stage, the $q$-th substrategy executes the $t(q)$-th action, that is, this defines $J_e(0^q1)[s]=I_i$.
Therefore, if {\sf Opp} wins, $P_e$ has no intersection with $\delta[V_i^{t(q)}]$.
This is because for any $x\in\delta[V^{t(q)}_i]$ has a name $z\in V_i^{t(q)}$, and therefore, $\varphi_e(z)$ extends $0^q1$ and $\psi_e(z)\not\in I_i=J_e(0^q1)$.

Consequently, we have $\delta[V^q]\subseteq Q^{t(q)}_i$, which forces that $\lambda^{k-q}(\delta[V^q])\leq 2^{q-t(q)}\cdot\varepsilon_{t(q)}\leq 2^{q-t(q)+1}$.
Eventually, we have $\lambda^{k-q}(\delta[V^q])=0$ as $t(q)$ tends to infinity.
Since $\delta[V^q]$ is a nonempty open subset of the convex set $P_e$, the condition $\lambda^{k-q}(\delta[V^q])=0$ implies that $P_e$ is at most $(k-q-1)$-dimensional.
Eventually, this construction forces that $P_e$ is zero-dimensional; hence, by convexity, $P_e$ has only one point.
Then, however, it must satisfy $P_e\subseteq\bigcup\{\delta[\varphi_e^{-1}[\rho]]:\rho\in 2^{p+1}\mbox{ and }\rho\not=\tau\}$ for some $\tau$.
Therefore, by compactness, this is witnessed at stage $s$, and then we answer yes to the question in (3).
This witnesses the failure of {\sf Opp}'s strategy as before, which contradicts our assumption.
\end{proof}

Suppose that $Z_0\circ Z_1\leq_W{\sf XC}_k$ via $H$ and $K=\langle K_0,K_1\rangle$, that is, given a pair $(T,J)$ of an a.o.u.~tree $T$ and a nonempty interval $J$, for any point $x$ of an at most $k$-dimensional convex closed set $H(T,J)$, we have $K_0(x,T,J)=p\in [T]$ and $K_1(x,T,J)\in J(p)$.
Then there is a computable function $\mathbb{N}\to\mathbb{N}$ such that $P_{f(e)}=H(T_e,J_e)$, $\varphi_{f(e)}=\lambda x.K_0(x,T_e,J_e)$ and $\psi_{f(e)}=\lambda x.K_1(x,T_e,J_e)$.
By Kleene's recursion theorem, there is $r\in\mathbb{N}$ such that $(P_{f(r)},\varphi_{f(r)},\psi_{f(r)})=(P_r,\varphi_r,\psi_r)$.
However, by the above claim, $(T_r,J_r)$ witnesses that {\sf Opp}'s strategy with $(P_r,\varphi_r,\psi_r)$ fails, which contradicts our assumption.
\end{proof}

\section*{Acknowledgement}
We are grateful to St\'ephane Le Roux for a fruitful discussion leading up to Theorems \ref{theo:simpleproducts} and \ref{theo:convexsort}.

\bibliographystyle{eptcs}
\bibliography{../../spieltheorie}

\end{document}